\documentclass[10pt]{article}

\usepackage{amsmath}
\usepackage{amsthm}
\usepackage{amsfonts}

\usepackage[total={6.15in,8.25in}]{geometry}
\usepackage{ifpdf}
\usepackage{graphicx,subfig}
\usepackage{epstopdf}
\ifpdf
  \DeclareGraphicsExtensions{.eps,.pdf,.png,.jpg}
\else
  \DeclareGraphicsExtensions{.eps}
\fi
\usepackage{color}

\usepackage{mdframed}
\mdfsetup{skipabove=\topskip,skipbelow=\topskip,nobreak=true}

\usepackage{booktabs}
\usepackage{multirow}
\usepackage{algorithm,algorithmic}

\newcommand{\rmnum}[1]{\romannumeral #1}
\newcommand{\Rmnum}[1]{\expandafter\@slowromancap\romannumeral #1@}
\renewcommand{\labelenumi}{(\roman{enumi})}

\usepackage{amssymb}

\theoremstyle{plain}
\newtheorem{theorem}{Theorem}[section]
\newtheorem{lemma}[theorem]{Lemma}
\newtheorem{proposition}[theorem]{Proposition}

\theoremstyle{definition}
\newtheorem{definition}{Definition}[section]

\theoremstyle{remark}
\newtheorem*{remark}{Remark}

\numberwithin{equation}{section}

\DeclareMathOperator{\sgn}{sgn}
\DeclareMathOperator{\supp}{supp}
\DeclareMathOperator{\dom}{dom}
\DeclareMathOperator{\dist}{dist}

\newcommand{\bfn}{\mathbf{n}}
\newcommand{\bfs}{\mathbf{s}}
\newcommand{\bft}{\mathbf{t}}
\newcommand{\bfu}{\mathbf{u}}
\newcommand{\bfv}{\mathbf{v}}

\newcommand{\bfx}{\mathbf{x}}
\newcommand{\bfy}{\mathbf{y}}
\newcommand{\bfz}{\mathbf{z}}

\newcommand{\bfA}{\mathbf{A}}
\newcommand{\bfB}{\mathbf{B}}
\newcommand{\bfI}{\mathbf{I}}

\newcommand{\bbR}{\mathbb{R}}
\newcommand{\calE}{\mathcal{E}}
\newcommand{\calF}{\mathcal{F}}

\newcommand{\sfK}{\mathsf{K}}
\newcommand{\sfM}{\mathsf{M}}
\newcommand{\sfN}{\mathsf{N}}

\newcommand{\ve}{\varepsilon}

\newcommand{\set}[1]{\left\{#1\right\}}
\newcommand{\abs}[1]{\left\vert#1\right\vert}
\newcommand{\norm}[1]{\left\Vert#1\right\Vert}

\allowdisplaybreaks

\renewcommand{\abstract}[1]{\textbf{Abstract.}~#1}
\providecommand{\keywords}[1]{\textbf{Keywords.}~#1}
\providecommand{\subjclass}[2]{\textbf{Mathematics subject classification (2010).}~#1}

\begin{document}


\title{\bf An iterative support shrinking algorithm for \\ $\ell_{p}$-$\ell_{q}$ minimization}
\author{Zhifang Liu, Yanan Zhao and Chunlin Wu
\thanks{Corresponding author. wucl@nankai.edu.cn}\\
School of Mathematical Sciences, Nankai University, Tianjin {\rm 300071}, China}

\date{\today}

\maketitle

\abstract{
We present an iterative support shrinking algorithm for $\ell_{p}$-$\ell_{q}$ minimization~($0 <p < 1 \leq q < \infty $). This algorithm guarantees the nonexpensiveness of the signal support set and can be easily implemented after being proximally linearized. The subproblem can be very efficiently solved due to its convexity and reducing size along iteration. We prove that the iterates of the algorithm globally converge to a stationary point of the $\ell_{p}$-$\ell_{q}$ objective function. In addition, we show a lower bound theory for the iteration sequence, which is more practical than the lower bound results for local minimizers in the literature.
}

\keywords{nonconvex nonsmooth regularization, non-Lipschitz optimization, support shrinking, sparse signal reconstruction, lower bound theory}

\subjclass{49M05,49K30,90C26,94A12,90C30}

\section{Introduction}
\label{sec:Intro}

Sparse reconstruction plays an important role in various applications such as signal and image processing, compressed sensing, model selection, variable selection, and many others~\cite{Donoho2006Compressed,Bruckstein2009From,Zou2008One,Lv2009unified,Tropp2010Computational}.
This problem can described as follows. Given an $\sfM  \times \sfN$ measurement matrix $ \bfA$  with $ \sfM < \sfN$, we consider to recover the sparse signal $ \bfx \in \bbR^{\sfN} $ from an observed signal $ \bfy = \bfA \bfx + \bfn \in \bbR^{\sfM}$, where $\bfn$ represents the measurement noise. There are many different types of noise. Two typical and important examples are Gaussian noise and heavy-tailed noise~\cite{Paredes2011Compressive}, which obey Gaussian distribution and heavier-than-Gaussian tails distribution \cite{Huber1981Robust}, respectively. To obtain the sparest solution, one naturally proposes to solve the following $\ell_0$ minimization problem
\[
	\min_{\bfx \in \bbR^{\sfN}} \norm{\bfx}_{0} \text{  subject to  } \norm{\bfA \bfx - \bfy}_{q} \leq \varepsilon,
\]
or its unconstrained counterpart
\begin{equation}\label{eq:l0-mini-prob}
    \min_{\bfx \in \bbR^{\sfN}}  \norm{\bfx}_{0} + \frac{1}{q \alpha}\norm{\bfA \bfx - \bfy}_{q}^{q}    ,
\end{equation}
where $\norm{\cdot}_{0}$ denotes the $\ell_0$ ``norm" that returns the number of nonzero entries of its argument, $\norm{\cdot}_{q}$ for $q \in [1, \infty)$ is the $\ell_{q}$~norm, and $ \alpha \in( 0, \infty)$ is a parameter that balances the regularization and the fidelity. The second term in~\eqref{eq:l0-mini-prob}, named as the fidelity term, is constructed using the noise distribution and the Maximum Likelihood principle. As well known, for Gaussian noise, people use the $\ell_{2}$ fidelity term ($q=2$). For heavy-tailed noise such as impulsive noise, the $\ell_{1}$ fidelity term ($q=1$) is a good choice~\cite{Dielman2005Least,Paredes2011Compressive}. Since the $\ell_0$ minimization is NP-hard~\cite{Nataraj1995Sparse}, numerous methods have been proposed to approximate it. Two common ways are to replace the $\ell_{0}$ ``norm" with the $\ell_1$ norm~\cite{Candes2006Robust,Donoho2006Compressed} and the $\ell_p$ quasi-norm~($0<p<1$) ~\cite{Donoho2006Compressed,Chartrand2008Iteratively,Foucart2009Sparsest,Daubechies2010Iteratively,Sun2012Recovery}, where the $\ell_p$ quasi-norm is defined as
$
    \norm{\bfx}_{p}= (\sum_{j=1}^{\sfN}\abs{x_{j}}^{p})^{1/p}
$.
In sparse reconstruction, the noncovex $\ell_p$ quasi-norm has some advantages~\cite{Chartrand2008Restricted,Foucart2009Sparsest,Sun2012Recovery,Lai2011Unconstrained} over the convex $\ell_1$ norm.

In this paper, we focus on the following $\ell_{p}$-$\ell_{q}$ minimization problem
\begin{equation}\label{eq:lp-mini-prob}
	\min_{\bfx \in \bbR^{\sfN}}
    \calE (\bfx):= \norm{\bfx}_{p}^{p} + \frac{1}{q \alpha} \norm{\bfA \bfx - \bfy}_{q}^{q} ,
\end{equation}
where $p \in(0,1)$, $q \in [1, \infty) $ and $ \alpha \in( 0, \infty)$. The objective function $\calE$ in~\eqref{eq:lp-mini-prob} is nonsmooth, nonconvex and non-Lipschtiz, which results in a great challenge for optimization. We now review some existing methods. As can been seen, most of them were designed for $\ell_{p}$-$\ell_{2}$ minimization.

One class of approaches is smoothing approximate methods~\cite{Chen2010Lower,Chen2012Smoothing,Chen2013Optimality,Bian2013Worst}, which are based on the special structure of the nonsmooth function $\calE$. By a smoothing function $\varphi(x,\theta)$ for the absolute value function $|x|$, the $\ell_{p},p \in(0,1)$ regularization term can be smoothed. Two choices of  $\varphi(x,\theta)$ in~\cite{Chen2012Smoothing,Chen2013Optimality,Bian2013Worst} are
\[
    \varphi_1(x,\theta) =
    \left\{
      \begin{array}{ll}
        |x| & \text{if~} |x| > \theta,  \\
        \frac{x^2}{2\theta} + \frac{\theta}{2}, & \text{if~} |x| \leq \theta,
      \end{array}
    \right.
\text{ and }
    \varphi_2(x,\theta) = \sqrt{x^2 + 4\theta^2}.
\]
Based on this technique, hybrid orthogonal matching pursuit-smoothing gradient (OMP-SG) method \cite{Chen2010Lower},
smoothing quadratic regularization~(SQR) algorithm~\cite{Bian2013Worst}, and smoothing trust region Newton method~\cite{Chen2013Optimality} have been proposed for $\ell_{p},p \in(0,1)$ regularized problems with smooth fidelity terms with convergence guarantee. They essentially reformulate the non-Lipschitz problem to be lipschitz ones by a smoothing parameter, which controls the approximate accuracy and need to be updated progressively to zero.

The second class of approaches is general iterative shrinkage-thresholding algorithms~(GISA) for $\ell_{p}$-$\ell_{2}$ minimization problem~\cite{Xu2012L12,Zuo2013Generalized,Bredies2015Minimization}. GISA was inspired by the great success of soft thresholding and iterative shrinkage-thresholding algorithms (ISTA)~\cite{Daubechies2004iterative,Beck2009fast} for convex $\ell_{1}$-$\ell_{2}$ minimization problem. Specifically, the general step of GISA is
\[
    \bfx^{(k+1)} = \mathcal{T}_{\alpha\beta}(\bfx^{(k)} + \beta \bfA^{T}(\bfy - \bfA \bfx^{(k)})),
\]
where $\beta > 0$ is an appropriate stepsize and $\mathcal{T}_{\alpha\beta}:\bbR^{\sfN}\to \bbR^{\sfN}$ is a shrinkage-thresholding operator. GISA is easy to implement, but it applies only to the case $q = 2$. Even for $q = 2$, the operator $\mathcal{T}_{\alpha\beta}$ have analytical expression only for $p = \frac{1}{2}$ and $p = \frac{2}{3}$~\cite{Xu2012L12,Krishnan2009Fast}. For a general $0<p<1$, the operator $\mathcal{T}_{\alpha\beta}$ needs to be computed via numerical methods~\cite{Zuo2013Generalized,Bredies2015Minimization}.

The third class of approaches is iterative reweighted minimization methods; see, e.g.~\cite{Gorodnitsky1997Sparse,Lai2011Unconstrained,Lai2013Improved,Candes2008Enhancing,Chen2014Convergence}. There are iterative reweighted least squares~(IRLS) and iterative reweighted $\ell_{1}$~(IRL1) minimization methods. One can refer to~\cite{Lu2014Iterative} for a systematic review. In~\cite{Lai2011Unconstrained,Lai2013Improved}, the authors considered a smoothed $\ell_{p}$-$\ell_{2}$ minimization
\[
    \sum_{j = 1}^{\sfN} (x_{j}^2 + \theta^2)^{p/2} + \frac{1}{2\alpha} \norm{\bfA \bfx - \bfy}_{2}^{2}
\]
and proposed IRLS algorithms to solve this approximate problem. In~\cite{Chen2014Convergence}, Chen and Zhou considered the following approximation to $\ell_{p}$-$\ell_{2}$ minimization
\[
    \sum_{j = 1}^{\sfN} (|x_{j}| + \theta)^{p} + \frac{1}{2\alpha} \norm{\bfA \bfx - \bfy}_{2}^{2}
\]
for some small $\theta > 0$. An IRL1 algorithm was proposed to slove this approximate problem. Both IRLS and IRL1 are stable. Actually reweighted methods reformulate the original non-Lipshitz $\ell_{p}$-$\ell_{2}$ to lipschitz ones by a de-singularizing parameter.

In this paper, we consider~\eqref{eq:lp-mini-prob} from a different perspective. We first obtain a proposition from the first order optimality condition. Motivated by this proposition, we propose an iterative algorithm with constraints on the support set of the signals. The core idea is to guarantee that the signal support set will not expand in the iterative procedure. After constraints elimination and proximally linearized, this algorithm can be easily implemented. The subproblem therein is convex and with reducing size along iteration. It is solved inexactly by alternating direction method of multipliers (ADMM). Furthermore, we establish the global convergence of the iterates to a stationary point of~\eqref{eq:lp-mini-prob}. We also prove a new lower bound theory for the iteration sequence, which is more practical than those lower bounds for local minimizers in the literature. Numerical examples show the good performance of our proposed algorithm for both $\ell_{p}$-$\ell_{2}$ and $\ell_{p}$-$\ell_{1}$ restoration.

The rest of this paper is organized as follows. In section~\ref{sec:Pre}, we give some basic notation and preliminaries. In section~\ref{sec:algo}, we describe the motivation, and propose our algorithms.
In section~\ref{sec:con-ana}, the convergence analysis is provided and the lower bound property of the iteration sequence is discussed. In section~\ref{sec:imlpem}, we give implementation details. The numerical experiments are shown in section~\ref{sec:num}. Section~\ref{sec:conclusion} concludes the paper.

\section{Some notations and preliminaries}
\label{sec:Pre}

Denote $I = \set{1,2,\dots, \sfM }$ and $J =\set{1,2,\dots,\sfN }$.
For a vector $\bfx \in \bbR^{\sfN}$, we refer to $x_{j}$ as its $j$th entry and denote the support set of $\bfx$ by
\[
\supp(\bfx):=\set{j \in J: x_{j} \neq 0}.
\]
We assume that all vectors are column  vectors.
For a matrix $\bfA \in \bbR^{\sfM \times \sfN}$, we write its $i$th row as $A_{i}^{T}$, which is the vector transpose of $A_{i} \in \bbR^{\sfN}$. Then we have
\[
\bfA =
\left[
\begin{array}{c}
  A_{1}^{T} \\
  \vdots \\
  A_{\sfM}^{T}
\end{array}
\right].
\]

Let $S$ be a subset of $J$.
We denote $\bfx_{S}$ be the subvector of $\bfx$  indexed by $S$, which consists of the nonzero entries of $\bfx$ when $S = \supp(\bfx)$. Similarly, we denote $\bfB = \bfA_{S}$ to be the column submatrix of $\bfA$ consisting of the columns indexed by $S$. Let $B_{i}^{T}$ be the $i$th row of $\bfB$, we have $B_{i} = (A_{i})_{S}$.

Define $\phi: [0,\infty) \to [0,\infty) $ by $\phi(x) = x^{p} (0 < p < 1)$. We state some useful properties for $\phi(\cdot)$.
\begin{proposition}
The function $\phi(\cdot)$ has the following properties:
\begin{enumerate}
\item $\phi(0) = 0$ and $\phi^{\prime}(x) = p x^{p-1}  > 0$ on $(0,\infty)$.
\item  $\phi(x)$ is concave and the following inequality holds,
\begin{equation}\label{eq:first-approx}
	\phi(y) \leq \phi(x) + \phi^{\prime}(x)(y - x),\; \forall x\in (0,\infty), y \in [0,\infty).
\end{equation}
\item For any $c > 0$, $\phi^{\prime}(x)$ is $ L_{c}$-Lipschitz continuous on $[c,\infty)$, i.e., there exists a constant $ L_{c} > 0$ determined by $c$, such that $\forall x,y \in [c, \infty)$,
\begin{equation}\label{eq:grad-Lip-cond}
	\abs{\phi^{\prime}(x) - \phi^{\prime}(y)} \leq L_{c} \abs{x-y}.
\end{equation}
\item The subdifferential of $\phi(\abs{x})$ at $x $ is given by
\[
     \partial \phi(\abs{x}) =
       \begin{cases}
         (-\infty,\infty), &  x = 0,\\
         \set{\sgn(x) \phi^{\prime}(\abs{x})} , &  x \neq 0,\\
       \end{cases}
\]
where $\sgn(x)$ is the signum function.
\end{enumerate}
\end{proposition}

Using the definition of $\phi(\cdot)$, we have
$
\norm{\bfx}_{p}^{p}=\sum_{j \in J}\phi(\abs{x_{j}}).
$
Thus the objective function $\calE$ in~\eqref{eq:lp-mini-prob} reads
\begin{equation}\label{eq:object-func}
	\calE(\bfx) = \sum_{j \in J}\phi(\abs{x_{j}})  + \frac{1}{q\alpha} \sum_{i \in I}\abs{A_{i}^{T} \bfx - y_{i} }^{q},
\end{equation}
which is bounded below and coercive.

Now, we drive the subdifferential of $\calE$ at $\bfx$.  For $1 < q < \infty$, by~\cite[Exercise 8.8 and Proposition 10.5]{Rockafellar1998Variational}, we get
\begin{equation}\label{eq:obj-func-subdiff-p2}
  \partial \calE(\bfx) = \partial \left(\sum_{j \in J}\phi(\abs{x_{j}})\right)
  +\frac{1}{q\alpha} \nabla \left( \sum_{i \in I}\abs{A_{i}^{T} \bfx - y_{i} }^{q} \right).
\end{equation}
where $\partial (\sum_{j \in J}\phi(\abs{x_{j}})) = \partial \phi(\abs{x_{1}}) \times \dots \times \partial \phi(\abs{x_{\sfN}})$.
For $q = 1$, we have checked the regularity requirement by~\cite[Corollary 10.9]{Rockafellar1998Variational}, indiciating
\begin{equation}\label{eq:obj-func-subdiff-p1}
  \partial \calE(\bfx) = \partial \left(\sum_{j \in J}\phi(\abs{x_{j}})\right)
  + \partial \left(\frac{1}{\alpha} \sum_{i \in I}\abs{A_{i}^{T} \bfx - y_{i} } \right).
\end{equation}

Throughout this paper, we say that  $\bfx^{\ast}$ is a stationary point of~\eqref{eq:lp-mini-prob} if  $\bfx^{\ast}$  satisfies
\begin{equation}\label{eq:first-opt-cond}
       0 \in \partial \calE(\bfx^{\ast}).
\end{equation}
If  $\bar{\bfx}$ is a local minimizer of~\eqref{eq:lp-mini-prob}, then the first-order optimality condition~\eqref{eq:first-opt-cond} holds.

\section{Motivation and the proposed algorithm}
\label{sec:algo}

\begin{proposition}\label{lem-motivation}
Given $\bfx \in \bbR^{\sfN}$. Suppose that $\bfx$ is sufficiently close to a local minimizer (or a stationary point) $\bfx^{\ast}$ of~\eqref{eq:lp-mini-prob}.
Then it holds that
   \begin{equation}\label{eq:motivation}
     x_{j}^{\ast} = 0, \; \forall j \in \Omega_{0} = J\setminus \supp(\bfx).
   \end{equation}
\end{proposition}

\begin{proof}
We prove~\eqref{eq:motivation} by contradiction. For the case of $1 < q < \infty$.
As $\bfx^{\ast}$ is a local minimizer (or a stationary point) of $\calE$, the condition~\eqref{eq:first-opt-cond} implies that for any $j \in J$, we have
     \[
       0 \in \partial \phi(\abs{x_{j}^{\ast}})
       + \frac{1}{\alpha} \left( \sum_{i \in I} \sgn(A_{i}^{T}\bfx^{\ast} - y_{i})\abs{A_{i}^{T}\bfx^{\ast} - y_{i}}^{q-1} A_{i}\right)_{j}.
     \]
     Assume that there is $j^{\prime} \in \Omega_{0} $ such that $x_{j^{\prime}}^{\ast} \neq 0$.
    Then we have
    \begin{equation}\label{eq:first-opt-cond-component-q2}
        0 = \sgn(x_{j^{\prime}}^{\ast}) \phi^{\prime}(\abs{x_{j^{\prime}}^{\ast}})
       + \frac{1}{\alpha}  \left( \sum_{i \in I}\sgn(A_{i}^{T}\bfx^{\ast} - y_{i}) \abs{A_{i}^{T}\bfx^{\ast} - y_{i}}^{q-1} A_{i}\right)_{j^{\prime}}.
    \end{equation}
The second term on the right side of~\eqref{eq:first-opt-cond-component-q2} is bounded. Since $\bfx$ is sufficiently close to $\bfx^{\ast}$, $x_{j^{\prime}}^{\ast}$ can be sufficiently close to $x_{j^{\prime}} = 0$. Then the equation~\eqref{eq:first-opt-cond-component-q2} is impossible to be true.  This is a contradiction.

For the case of $q = 1$, the condition~\eqref{eq:first-opt-cond} can be rewritten as  for any $j \in J$,
\begin{equation}\label{eq:first-opt-cond-component-q1}
    0 \in \partial \phi(\abs{x_{j}^{\ast}})
       + \frac{1}{\alpha}  \left( \sum_{i \in I} \partial |\cdot|(A_{i}^{T}\bfx^{\ast} - y_{i}) A_{i}\right)_{j}.
\end{equation}
Using the boundedness of the second term on the right side of~\eqref{eq:first-opt-cond-component-q1}, we can prove the results similarly.
\end{proof}

Motivated by Proposition~\ref{lem-motivation}, we propose to solve the problem~\eqref{eq:lp-mini-prob} by an iteration process, which generates a sequence with nonincreasing support set. Suppose that $\bfx^{(k)}$ is an approximate solution in the $k$th iteration. In the next iteration, we minimize the objective function with the restriction of zero entries outside the support set of  $\bfx^{(k)}$. This idea yields the following iterative support shrinking algorithm (ISSA) for solving~\eqref{eq:lp-mini-prob}.

\begin{mdframed}[frametitle = {ISSA: Iterative Support Shrinking Algorithm}, frametitlerule = true]

\noindent{\bf Initialization:} Select $\bfx^{(0)} \in \bbR^{\sfN}$.

\noindent{\bf Iteration:} For $k = 0, 1, \ldots$ until convergence:
\begin{enumerate}
\renewcommand{\labelenumi}{ \arabic{enumi}.}
  \item Set $S^{(k)}  = \supp(\bfx^{(k)})$.
  \item Compute $\bfx^{(k+1)}$ by solving
\begin{equation}\label{eq:constrained-lp-lq-prob}
    \left\{
    \begin{aligned}
      &  \min_{\bfx \in \bbR^{\sfN}} \sum_{ j \in S^{(k)} } \phi( |x_{j}|)
	+ \frac{1}{q\alpha} \| \bfA \bfx - \bfy \|_{q}^{q}, \\
      & \text{s. t. } x_{j} = 0, \; \forall j \in \Omega_{0}^{(k)} = J \setminus  S^{(k)}.
    \end{aligned}
    \right. \tag{$\mathcal{P}_{x}$}
\end{equation}
\end{enumerate}

\end{mdframed}

In fact,  the problem~\eqref{eq:constrained-lp-lq-prob} amounts to minimize the objective function respect to only $\# S^{(k)}$ entries of $\bfx$, with the remaining components being null. Note that $S^{(k)}$ is the support of $\bfx^{(k)}$. Given a vector $\bfx \in \bbR^{\sfN}$  with $\supp(\bfx) \subseteq S^{(k)}$, we let  $ \bfz^{(k)} = \bfx_{S^{(k)}}^{(k)} $, $\bfB^{(k)} =  \bfA_{S^{(k)}} $ and $\bfz = \bfx_{S^{(k)}}$. It follows that
\begin{equation}\label{eq:basic-supp-eq}
    \bfB^{(k)} \bfz = \bfA \bfx,\; \| \bfz - \bfz^{(k)} \|_{2}= \| \bfx - \bfx^{(k)} \|_{2}.
\end{equation}
These relationships help to reformulate the problem~\eqref{eq:constrained-lp-lq-prob} to an unconstrained problem with $\bfz$ as the unknowns. At the same time, each term $\phi( |x_{j}|), j \in S^{(k)}$ can be linearized at $|x_{j}^{(k)}| \neq 0$. Together with a proximal technique, we present an iterative support shrinking algorithm with proximal linearization (ISSAPL) to solve~\eqref{eq:lp-mini-prob}.

\begin{mdframed} [frametitle = {ISSAPL: Iterative Support Shrinking Algorithm with Proximal Linearization}, frametitlerule = true]

\noindent {\bf Initialization:}  Select $\bfx^{(0)}\in \bbR^{\sfN}$ and $\beta > 0$.

\noindent {\bf Iteration:} For $k = 0, 1, \ldots$ until convergence:
\begin{enumerate}
\renewcommand{\labelenumi}{ \arabic{enumi}.}
\renewcommand{\labelenumii}{ \arabic{enumi}.\arabic{enumii}.}
  \item Set $S^{(k)}  = \supp(\bfx^{(k)})$.
  \item Generate $\bfx^{(k+1)}$ as follows:
    \begin{enumerate}
      \item  Set $ \bfz^{(k)} = \bfx_{S^{(k)}}^{(k)} $ and $\bfB^{(k)} =  \bfA_{S^{(k)}} $.
      \item Compute $\hat{\bfz}^{(k+1)}$ by solving
\begin{equation}\label{eq:prob-z}
	 \min_{\bfz } \hat{\calE}^{(k)}(\bfz)=
    \sum_{ j \in S^{(k)} } \phi^{\prime}(|x_{j}^{(k)}|) |z_{j}|	+ \frac{1}{q\alpha} \| \bfB^{(k)} \bfz - \bfy \|_{q}^{q} + \frac{\beta}{2}\| \bfz - \bfz^{(k)} \|_{2}^{2}.
\tag{$\mathcal{P}_{z}$}
\end{equation}
      \item Set
\begin{equation}\label{eq:z-to-x}
  x_{j}^{(k+1)} =
\begin{cases}
	0, & j \in \Omega_{0}^{(k)} = J \setminus S^{(k)},\\
    \hat{z}_{j}^{(k+1)}, & j \in S^{(k)}.
\end{cases}
\end{equation}
    \end{enumerate}
\end{enumerate}

\end{mdframed}

The problem~\eqref{eq:prob-z} in ISSAPL has an unique optimal solution due to strong convexity of $\hat{\calE}^{(k)}$. Although the problem~\eqref{eq:prob-z} is a convex optimization problem, it needs to be solved by iteration. In practical, we solve the ~\eqref{eq:prob-z} inexactly. Now we present our inexact iterative support shrinking algorithm with proximal linearization (InISSAPL) to solve~\eqref{eq:lp-mini-prob}.

\begin{mdframed} [frametitle = {InISSAPL: Inexact Iterative Support Shrinking Algorithm with Proximal Linearization}, frametitlerule = true]

\noindent {\bf Initialization:}  Select $\bfx^{(0)}\in \bbR^{\sfN}$, $\beta > 0$ and $ 0 \leq \ve < 1$.

\noindent {\bf Iteration:} For $k = 0, 1, \ldots$ until convergence:
\begin{enumerate}
\renewcommand{\labelenumi}{ \arabic{enumi}.}
\renewcommand{\labelenumii}{ \arabic{enumi}.\arabic{enumii}.}
  \item Set $S^{(k)}  = \supp(\bfx^{(k)})$.
  \item Generate $\bfx^{(k+1)}$ as follows:
    \begin{enumerate}
      \item  Set $ \bfz^{(k)} = \bfx_{S^{(k)}}^{(k)} $ and $\bfB^{(k)} =  \bfA_{S^{(k)}} $.
      \item Find $	\hat{\bfz}^{(k+1)} \approx \arg\min_{\bfz } \hat{\calE}^{(k)}(\bfz)$ and     $\hat{\bfu}^{(k+1)} \in \partial \hat{\calE}(\hat{\bfz}^{(k+1)})$, such that
\begin{equation}\label{eq:sub-opt-cond}
  \| \hat{\bfu}^{(k+1)} \|_{2} \leq  \frac{\beta }{2} \ve  \| \hat{\bfz}^{(k+1)} - \bfz^{(k)}\|_{2}.
\end{equation}
      \item Set
\begin{equation}\label{eq:z-to-x-in}
  x_{j}^{(k+1)} =
\begin{cases}
	0, & j \in \Omega_{0}^{(k)} = J \setminus S^{(k)},\\
    \hat{z}_{j}^{(k+1)}, & j \in S^{(k)}.
\end{cases}
\end{equation}
    \end{enumerate}
\end{enumerate}

\end{mdframed}

\begin{remark}
The condition~\eqref{eq:sub-opt-cond} in InISSAPL is motivated by~\cite{Attouch2013Convergence}. It corresponds to an inexact optimality condition and a guide to select the approximate solution for \eqref{eq:prob-z}. Due to the strong convexity of the problem \eqref{eq:prob-z}, it can be solved to any given accuracy. Therefore, the condition~\eqref{eq:sub-opt-cond} in InISSAPL can hold, as long as the problem \eqref{eq:prob-z} is sufficiently solved.
\end{remark}

We have some useful representations of $\hat{\bfu}^{(k+1)}$. Since $\hat{\bfu}^{(k+1)} \in \partial \hat{\calE}(\hat{\bfz}^{(k+1)}) $, we have
\begin{equation}\label{eq:prob-z-in-opt-cond}
\begin{split}
   \hat{\bfu}^{(k+1)} & \in \partial \left( \sum_{ j \in S^{(k)} } \phi^{\prime}(|z_{j}^{(k)}|) |\hat{z}_{j}^{(k+1)}| \right)
   + \beta(\hat{\bfz}^{(k+1)} - \bfz^{(k)} ) \\
    &\phantom{=;} + \frac{1}{q\alpha} \partial \left(\sum_{i \in I} \abs{(B_{i}^{(k)})^{T} \hat{\bfz}^{(k+1)} - y_{i} }^{q}  \right).
\end{split}
\end{equation}
Then for any $j \in S^{(k)}$ and $i \in I$, there are $\xi_{j} \in \partial |\cdot|(\hat{z}_{j}^{(k+1)}) = \partial |\cdot|(x_{j}^{(k+1)})$ and
$\eta_{i} =\partial |\cdot|((B_{i}^{(k)})^{T} \hat{\bfz}^{(k+1)} - y_{i} )  = \partial |\cdot|(A_{i}^{T} \bfx^{(k+1)} - y_{i} )$,
such that when $1 < q < \infty$,
\begin{equation}\label{eq:prob-z-sub-diff-q2}
\begin{split}
    \hat{u}_{j}^{(k+1)}  & = \xi_{j} \phi^{\prime}(|z_{j}^{(k)}|)  +\beta(\hat{z}_{j}^{(k+1)} - z_{j}^{(k)})\\
    & \phantom{=;} + \frac{1}{\alpha}\left(\sum_{i\in I} \sgn((B_{i}^{(k)})^{T} \hat{\bfz}^{(k+1)} - y_{i} ) \abs{(B_{i}^{(k)})^{T} \hat{\bfz}^{(k+1)} - y_{i}}^{q-1} B_{i}^{(k)}\right)_{j} \\
     & = \xi_{j} \phi^{\prime}(|x_{j}^{(k)}|)  +\beta(x_{j}^{(k+1)} - x_{j}^{(k)})  \\
& \phantom{=;} + \frac{1}{\alpha}\left(\sum_{i\in I} \sgn(A_{i}^{T} \bfx^{(k+1)} - y_{i} ) \abs{A_{i}^{T} \bfx^{(k+1)} - y_{i}}^{q-1} A_{i}\right)_{j},
\end{split}
\end{equation}
and when $q=1$,
\begin{equation}\label{eq:prob-z-sub-diff-q1}
\begin{split}
   \hat{u}_{j}^{(k+1)}
    & = \xi_{j} \phi^{\prime}(|z_{j}^{(k)}|)   +\beta(\hat{z}_{j}^{(k+1)} - z_{j}^{(k)}) + \frac{1}{\alpha}\left(\sum_{i\in I} \eta_{i} B_{i}^{(k)}\right)_{j}\\
     & = \xi_{j} \phi^{\prime}(|x_{j}^{(k)}|)  +\beta(x_{j}^{(k+1)} - x_{j}^{(k)}) + \frac{1}{\alpha}\left(\sum_{i\in I} \eta_{i} A_{i}\right)_{j}.
\end{split}
\end{equation}

\section{Convergence analysis}
\label{sec:con-ana}

In this section, we establish the global convergence result of the sequence by the proposed InISSAPL. These results also hold for ISSAPL.

From the iteration process of InISSAPL, we can see that it generates a nonincreasing sequence of support set. A basic lemma for $\set{S^{(k)}}$ is showed in the following.

\begin{lemma}\label{lem-finite-num-iter}
The sequence $\set{S^{(k)}}$ converges in a finite number of iterations, i.e., there exists an integer $ \sfK > 0$ such that if $k \geq \sfK$, then $S^{(k)} \equiv  S^{(\sfK)}$.
\end{lemma}

\begin{proof}
Since
\[
  J \supseteq S^{(0)} \supseteq \cdots \supseteq S^{(k)} \supseteq \cdots,
\]
 $\set{S^{(k)}}$ converges in a finite number of iterations.
\end{proof}

Lemma~\ref{lem-finite-num-iter} plays a key role in the convergence analysis for the three vector sequences, $\set{\bfx^{(k)}}$, $\set{\bfz^{(k)}}$ and $\set{\hat{\bfz}^{(k)}}$, generated by InISSAPL. Note that $\bfz^{(k)} = \bfx_{S^{(k)}}^{(k)}$. By~\eqref{eq:z-to-x-in}, $\hat{\bfz}^{(k)}$ and $\bfx^{(k)}$ have exactly the same nonzero entries. From Lemma~\ref{lem-finite-num-iter}, we can claim that after a certain $\sfK$ numbers of iteration, the support of $\bfx^{(k)}$ is fixed, i.e., if $k \geq \sfK$, then $\supp(\bfx^{(k)}) = S^{(\sfK)}$, from which we can directly obtain that
\[
	\bfz^{(k)} = \bfx_{S^{(\sfK)}}^{(k)} = \hat{\bfz}^{(k)}, \forall k > \sfK.
\]

In the next, we establish the global convergence of the sequence $\set{\bfx^{(k)}}$.
For the convenience of description, we introduce an auxiliary function
\begin{equation}\label{eq:prox-linear-approx-obj}
\begin{split}
   \calF^{(k)}(\bfx) & =  \sum_{ j \in S^{(k)} } \phi(|x_{j}^{(k)}|)
	+ \phi^{\prime}(|x_{j}^{(k)}|) \left(|x_{j}| - |x_{j}^{(k)}|\right)
	+ \frac{1}{q\alpha} \| \bfA \bfx - \bfy \|_{q}^{q} \\
     & \phantom{=;} + \frac{\beta}{2}\| \bfx - \bfx^{(k)} \|_{2}^{2}.
\end{split}	
\end{equation}

\begin{lemma}\label{lem-bound-sufficient-decrease}
For any $\beta > 0$ and $0 \leq \ve < 1$, let $\set{\bfx^{(k)}}$ be a sequence generated by {\rm InISSAPL}. Then
\begin{enumerate}
\item The sequence $\set{\calE(\bfx^{(k)})}$ is nonincreasing and satisfies
    \begin{equation}\label{eq:sufficient-decrease-condition}
    \calE(\bfx^{(k+1)})+ \frac{\beta}{2}(1-\ve)\| \bfx^{(k+1)} - \bfx^{(k)} \|_{2}^{2} \leq \calE(\bfx^{(k)}).
    \end{equation}
\item The sequence $\set{\bfx^{(k)}}$  is bounded and satisfies $\lim_{k \to \infty} \| \bfx^{(k+1)} - \bfx^{(k)} \|_{2} = 0$.
\end{enumerate}
\end{lemma}

\begin{proof}
For the case of $1 < q < \infty$. Due to the fact that $\phi(0) = 0$, we have
\begin{equation}\label{eq:mm-eq1}
\begin{split}
	\calF^{(k)}(\bfx^{(k)}) & =\sum_{ j \in S^{(k)} } \phi(|x_{j}^{(k)}|) + \frac{1}{q\alpha} \| \bfA \bfx^{(k)} - \bfy \|_{q}^{q}  \\
	& =\sum_{ j \in J } \phi(|x_{j}^{(k)}|) + \frac{1}{q\alpha} \| \bfA \bfx^{(k)} - \bfy \|_{q}^{q} = \calE(\bfx^{(k)}).
\end{split}
\end{equation}
When $\bfx \in \bbR^{\sfN}$ and $\supp(\bfx) \subseteq S^{(k)}$, we obtain
\begin{equation}\label{eq:mm-eq2}
\begin{split}
	\calF^{(k)}(\bfx) &= \sum_{ j \in \supp(\bfx) } \phi(|x_{j}^{(k)}|)
	+ \phi^{\prime}(|x_{j}^{(k)}|) \left(|x_{j}| - |x_{j}^{(k)}|\right) \\
    &\phantom{=;} + \sum_{ j \in S^{(k)}\setminus \supp(\bfx) } \phi(|x_{j}^{(k)}|)
	- \phi^{\prime}(|x_{j}^{(k)}|) |x_{j}^{(k)}| \\
    &\phantom{=;} + \frac{1}{q\alpha} \| \bfA \bfx - \bfy \|_{q}^{q}
	+ \frac{\beta}{2}\| \bfx - \bfx^{(k)} \|_{2}^{2}
\\
[~\text{by~\eqref{eq:first-approx}}~]~  &\geq  \sum_{ j \in S^{(k)} } \phi(|x_{j}|) 	+ \frac{1}{q\alpha} \| \bfA \bfx - \bfy \|_{q}^{q} + \frac{\beta}{2}\| \bfx - \bfx^{(k)} \|_{2}^{2} \\
	&=  \sum_{ j \in J } \phi(|x_{j}|) 	+ \frac{1}{q\alpha} \| \bfA \bfx - \bfy \|_{q}^{q}
	+ \frac{\beta}{2}\| \bfx - \bfx^{(k)} \|_{2}^{2} \\
	&= \calE(\bfx)+ \frac{\beta}{2}\| \bfx - \bfx^{(k)} \|_{2}^{2}.
\end{split}
\end{equation}

The subdifferential of $\calF^{(k)}$ at $\bfx$ is defined as
\[
\begin{split}
	\partial \calF^{(k)}(\bfx) &= \partial \left(\sum_{ j \in S^{(k)} } \phi^{\prime}(|x_{j}^{(k)}|) |x_{j}| \right)
+ \beta( \bfx - \bfx^{(k)})\\
	& \phantom{=;} +\frac{1}{q\alpha} \nabla \left( \sum_{i \in I}\abs{A_{i}^{T} \bfx - y_{i} }^{q} \right) \\
	& = \partial \left( \sum_{ j \in S^{(k)} }  \phi^{\prime}(|x_{j}^{(k)}|) |x_{j}| \right)
	+ \beta( \bfx - \bfx^{(k)}) \\
    & \phantom{=;} + \frac{1}{\alpha} \sum_{i\in I} \sgn(A_{i}^{T} \bfx- y_{i} ) \abs{A_{i}^{T} \bfx - y_{i}}^{q-1} A_{i}.
\end{split}
\]
Let $\bfu^{(k+1)} = (u_{1}^{(k+1)}, \ldots, u_{\sfN}^{(k+1)})^{T} $, where
\[
	u_{j}^{(k+1)} =
\begin{cases}
	\hat{u}_{j}^{(k+1)}, & j \in S^{(k)} \\
	\frac{1}{\alpha}\left(\sum_{i\in I} \sgn(A_{i}^{T} \bfx^{(k+1)} - y_{i} ) \abs{A_{i}^{T} \bfx^{(k+1)} - y_{i}}^{q-1} A_{i}\right)_{j}, & \text{otherwise},
\end{cases}	
\]
with $\hat{u}_{j}^{(k+1)}$ in \eqref{eq:prob-z-sub-diff-q2}.
Then $\bfu^{(k+1)} \in \partial \calF^{(k)}(\bfx^{(k+1)}) $. Since for any $ j \in J \setminus  S^{(k)}$, $x_{j}^{(k+1)} = x_{j}^{(k)} = 0$, we have
\begin{equation}\label{eq:mm-eq3}
\begin{split}
	\langle \bfu^{(k+1)} , \bfx^{(k)} - \bfx^{(k+1)} \rangle
	& = \sum_{j \in S^{(k)} } u_{j}^{(k+1)}(x_{j}^{(k)} - x_{j}^{(k+1)})\\
	 & = \sum_{j \in S^{(k)}} \hat{u}_{j}^{(k+1)}(z_{j}^{(k)} - \hat{z}_{j}^{(k+1)}) \\
	&\geq - \|\hat{\bfu}^{(k+1)} \|_{2} \| \bfz^{(k)} - \hat{\bfz}^{(k+1)} \|_{2}\\
[~\text{by~\eqref{eq:sub-opt-cond}}~]~	&\geq - \frac{\beta}{2}\ve \| \bfz^{(k)} - \hat{\bfz}^{(k+1)} \|_{2}^{2}\\
[~\text{by~\eqref{eq:basic-supp-eq}}~]~    &= - \frac{\beta}{2}\ve \| \bfx^{(k)} - \bfx^{(k+1)} \|_{2}.
\end{split}
\end{equation}

Putting~\eqref{eq:mm-eq1}, \eqref{eq:mm-eq2} and~\eqref{eq:mm-eq3} together, we obtain
\[
\begin{split}
	\calE(\bfx^{(k)}) = \calF^{(k)}(\bfx^{(k)}) & \geq \calF^{(k)}(\bfx^{(k+1)}) + \langle \bfu^{(k+1)} , \bfx^{(k)} - \bfx^{(k+1)} \rangle \\
	& \geq \calF^{(k)}(\bfx^{(k+1)}) - \frac{\beta}{2}\ve \| \bfx^{(k+1)} - \bfx^{(k)} \|_{2}\\
	& \geq \calE(\bfx^{(k+1)})+ \frac{\beta}{2}(1-\ve)\| \bfx^{(k+1)} - \bfx^{(k)} \|_{2}^{2}.
\end{split}	
\]

With the fact that $\calE(\bfx)$ is bounded from below and $ \frac{\beta}{2}(1-\ve) > 0$, it follows that $\set{\calE(\bfx^{(k)})}$ is nonincreasing and converges to a finite value as $k \to \infty$. Thus $\lim_{k\to \infty} \| \bfx^{(k+1)} - \bfx^{(k)} \|_{2} = 0$.

Because $\calE(\bfx)$ is coercive, we know that $\set{\bfx^{(k)}}$ is bounded.

For the case of $q =1$, the subdifferential of $\calF^{(k)}$ at $\bfx$ is given by
\[
\begin{split}
	\partial \calF^{(k)}(\bfx) &= \partial \left(\sum_{ j \in S^{(k)} } \phi^{\prime}(|x_{j}^{(k)}|) |x_{j}| \right)
    + \beta( \bfx - \bfx^{(k)})
	+ \partial \left(\frac{1}{\alpha} \sum_{i \in I}\abs{A_{i}^{T} \bfx - y_{i} } \right)  \\
	& = \partial \left( \sum_{ j \in S^{(k)} }  \phi^{\prime}(|x_{j}^{(k)}|) |x_{j}| \right)
    + \beta( \bfx - \bfx^{(k)})
	+ \frac{1}{\alpha} \sum_{i\in I} \partial |\cdot|(A_{i}^{T} \bfx- y_{i} ) A_{i}.
\end{split}
\]
Let $\bfu^{(k+1)} = (u_{1}^{(k+1)}, \ldots, u_{\sfN}^{(k+1)})^{T} $ with
\[
	u_{j}^{(k+1)} =
\begin{cases}
	\hat{u}_{j}^{(k+1)}, & j \in S^{(k)} \\
	\frac{1}{\alpha}\left(\sum_{i\in I} \eta_{i} A_{i} \right)_{j},  & \text{otherwise},
\end{cases}	
\]
where $\hat{u}_{j}^{(k+1)}$ is as in \eqref{eq:prob-z-sub-diff-q1} and $\eta_{i} \in \partial |\cdot|(A_{i}^{T} \bfx- y_{i} )$.
Then $\bfu^{(k+1)} \in \partial \calF^{(k)}(\bfx^{(k+1)}) $.
In a similar way, we can prove that (\rmnum{1})(\rmnum{2}) holds.
\end{proof}

Recall the results of Lemma~\ref{lem-finite-num-iter}, we now focus on the iteration number $k \geq K$ to get the convergence of the sequence $\set{\bfx^{(k)}}$.
Then the entries of $\hat{\bfu}^{(k+1)}$ in \eqref{eq:prob-z-in-opt-cond} can be written as for any $ j \in S^{(k)}$, when $1<q<\infty$,
\begin{equation}\label{eq:prob-z-sub-diff-q2-K}
\begin{split}
    \hat{u}_{j}^{(k+1)}
     & = \sgn(x_{j}^{(k+1)}) \phi^{\prime}(|x_{j}^{(k)}|)   + \beta(x_{j}^{(k+1)} - x_{j}^{(k)})  \\
& \phantom{=;} + \frac{1}{\alpha}\left(\sum_{i\in I} \sgn(A_{i}^{T} \bfx^{(k+1)} - y_{i} ) \abs{A_{i}^{T} \bfx^{(k+1)} - y_{i}}^{q-1} A_{i}\right)_{j},
\end{split}
\end{equation}
and when $q=1$,
\begin{equation}\label{eq:prob-z-sub-diff-q1-K}
	\hat{u}_{j}^{(k+1)} =  \sgn(x_{j}^{(k+1)}) \phi^{\prime}(|x_{j}^{(k)}|)   + \beta(x_{j}^{(k+1)} - x_{j}^{(k)}) + \frac{1}{\alpha}\left(\sum_{i\in I} \eta_{i} A_{i}\right)_{j}.
\end{equation}
The condition~\eqref{eq:sub-opt-cond} reads
\begin{equation}\label{eq:inexact-opt-cond}
	\| \hat{\bfu}^{(k+1)} \|_{2} \leq \frac{\beta}{2} \ve\| \hat{\bfz}^{(k+1)} - \bfz^{(k)} \|_{2}
	=\frac{\beta}{2} \ve \| \bfx^{(k+1)} - \bfx^{(k)} \|_{2},
\end{equation}
for $1 \leq q < \infty$.

The following is a bound theory on the iteration sequence, which is important to establish the convergence of  $\set{\bfx^{(k)}}$.
\begin{theorem}\label{lem-bound-theory}
There are $ 0 < c < C < \infty$ such that
\begin{equation}\label{eq:bound-theory}
	either \quad x_{j}^{(k)} = 0  \quad or \quad c \leq |x_{j}^{(k)}| \leq C, \; \forall j \in J, \; \forall k \geq \sfK.
\end{equation}
\end{theorem}

\begin{proof}
From Lemma~\ref{lem-finite-num-iter}, for any $j \in S^{(\sfK)}$ and $k \geq \sfK$, $x_{j}^{(k)} \neq 0$. We now prove by contradiction that $|x_{j}^{(k)}| $ has nonzero lower and upper bound for any $j \in S^{(\sfK)}, \forall k \geq \sfK$.

For the case of $1 < q < \infty$, assume there exists $j^{\prime} \in S^{(\sfK)}$ such that
$
  x_{j^{\prime}}^{(k)} \neq 0 \text{ and } \lim_{k\to \infty} x_{j^{\prime}}^{(k)} = 0.
$
Note that, if necessary, we can pass to a subsequence of $x_{j^{\prime}}^{(k)}$. By letting
\[
\zeta_{j^{\prime}} = \beta(x_{j^{\prime}}^{(k+1)} - x_{j^{\prime}}^{(k)}) + \frac{1}{\alpha}\left(\sum_{i\in I} \sgn(A_{i}^{T} \bfx^{(k+1)} - y_{i} ) \abs{A_{i}^{T} \bfx^{(k+1)} - y_{i}}^{q-1} A_{i}\right)_{j^{\prime}},
\]
we have
\begin{equation}\label{eq:bound-entry-triangle-K}
  \abs{\phi^{\prime}(|x_{j^{\prime}}^{(k)}|)} \leq \abs{\hat{u}_{j^{\prime}}^{(k+1)}} + \abs{\zeta_{j^{\prime}}}
\end{equation}
according to~\eqref{eq:prob-z-sub-diff-q2-K}.
It follows from the boundness of $\set{\bfx^{(k)}}$ (Lemma~\ref{lem-bound-sufficient-decrease}) that $\abs{\zeta_{j^{\prime}}}$ is bounded.
The condition~\eqref{eq:inexact-opt-cond} implies that $\abs{\hat{u}_{j^{\prime}}^{(k+1)}}$ is also bounded. Thus the equation~\eqref{eq:bound-entry-triangle-K} is impossible to hold when $k \to \infty$.

For the case of $q = 1$, by letting
\[
\zeta_{j^{\prime}}=  \beta(x_{j^{\prime}}^{(k+1)} - x_{j^{\prime}}^{(k)}) +  \frac{1}{\alpha}\left(\sum_{i\in I} \eta_{i} A_{i}\right)_{j^{\prime}},
\]
we have
\[
  \abs{\phi^{\prime}(|x_{j^{\prime}}^{(k)}|)} \leq \abs{\hat{u}_{j^{\prime}}^{(k+1)}} + \abs{\zeta_{j^{\prime}}},
\]
according to~\eqref{eq:prob-z-sub-diff-q1-K}.
Using the boundness of the right-hand side, we can prove the results similarly.
\end{proof}

Compared with the lower bound theory for local minimizers in the literature~\cite{Chen2010Lower}, the bound theory in Theorem~\ref{lem-bound-theory} is for the iterative sequence and more practical.
 Theorem~\ref{lem-bound-theory} indicates that when $k \geq \sfK$, there exists $L_{c} > 0$, such that for any $ j \in S^{(\sfK)} $,
\begin{equation}\label{eq:grad-Lip-cond-K}
	\abs{ \phi^{\prime}(|x_{j}^{(k+1)}|) - \phi^{\prime}(|x_{j}^{(k)}|) }
	\leq L_{c}\abs{|x_{j}^{(k+1)}| - |x_{j}^{(k)}|}
	\leq L_{c}\abs{x_{j}^{(k+1)} - x_{j}^{(k)}}.
\end{equation}

We now derive a subgradient lower bound for the iterates gap.

\begin{lemma}\label{lem-relative-error-condition}
For each $k \geq K$, there exists $\bfv^{(k+1)}  \in \partial \calE(\bfx^{(k+1)})$ such that
\begin{equation}\label{eq:relative-error-condition}
	\| \bfv^{(k+1)} \|_{2}
	\leq \left(L_{c} + \frac{\beta}{2} (\ve + 2) \right) \|\bfx^{(k+1)} - \bfx^{(k)} \|_{2}.
\end{equation}
\end{lemma}

\begin{proof}
For the case of $1 < q < \infty$, denote
\[
\begin{split}
	v_{j}^{(k+1)} &= \sgn(x_{j}^{(k+1)})\phi^{\prime}(|x_{j}^{(k+1)}|)  \\
	& \phantom{=;} + \frac{1}{\alpha} \left(\sum_{i\in I} \sgn(A_{i}^{T} \bfx^{(k+1)} - y_{i} ) \abs{A_{i}^{T} \bfx^{(k+1)} - y_{i}}^{q-1} A_{i}\right)_{j}, \forall j \in S^{(\sfK)}, \\
	v_{j}^{(k+1)} &= 0, \forall j \in J \setminus S^{(\sfK)};
\end{split}
\]
\[
\begin{split}
	\hat{v}_{j}^{(k+1)} &= \sgn(x_{j}^{(k+1)})\phi^{\prime}(|x_{j}^{(k)}|) \\
	& \phantom{=;} + \frac{1}{\alpha} \left(\sum_{i\in I} \sgn(A_{i}^{T} \bfx^{(k+1)} - y_{i} ) \abs{A_{i}^{T} \bfx^{(k+1)} - y_{i}}^{q-1} A_{i}\right)_{j}, \forall j \in S^{(\sfK)}, \\
	\hat{v}_{j}^{(k+1)} &= 0, \forall j \in J \setminus S^{(\sfK)}.
\end{split}
\]
Since $\partial \phi(|0|) = (-\infty,\infty)$, we have $\bfv^{(k+1)} = (v_{1}^{(k+1)},\dots,v_{\sfN}^{(k+1)})^{\top} \in \partial \calE(\bfx^{(k+1)})$ and
\begin{equation}\label{eq:subdiff-bound-eq1}
\begin{split}
	\| \hat{\bfv}^{(k+1)} \|_{2} &= \sqrt{\sum_{j \in S^{(\sfK)}} |\hat{v}_{j}^{(k+1)}|^{2}}
	 = \sqrt{\sum_{j \in S^{(\sfK)}} |\hat{u}_{j}^{(k+1)} - \beta (x_{j}^{(k+1)} - x_{j}^{(k)}) |^{2}} \\
	& \leq \| \hat{\bfu}^{(k+1)} \|_{2} + \beta \|\bfx^{(k+1)} - \bfx^{(k)} \|_{2} \\
[~\text{by~\eqref{eq:inexact-opt-cond}}~]~	& \leq \frac{\beta}{2}( \ve + 2 ) \|\bfx^{(k+1)} - \bfx^{(k)} \|_{2}.
\end{split}
\end{equation}
Form~\eqref{eq:grad-Lip-cond-K}, it follows that
\begin{equation}\label{eq:subdiff-bound-eq2}
\begin{split}
	& \phantom{=;}\| \bfv^{(k+1)} - \hat{\bfv}^{(k+1)} \|_{2} \\
    &= \sqrt{\sum_{j \in S^{(\sfK)}}\abs{  \sgn(x_{j}^{(k+1)})\phi^{\prime}(|x_{j}^{(k+1)}|) - \sgn(x_{j}^{(k+1)})\phi^{\prime}(|x_{j}^{(k)}|)  }^{2}}\\
[~\text{by~\eqref{eq:grad-Lip-cond-K}}~]~	 & \leq L_{c}\sqrt{\sum_{j \in S^{(\sfK)}} \abs{ |x_{j}^{(k+1)}| - |x_{j}^{(k)}| }^{2} }\\
    &\leq L_{c}\sqrt{\sum_{j \in S^{(\sfK)}} \abs{ x_{j}^{(k+1)} - x_{j}^{(k)} }^{2} } \\
 & = L_{c} \|\bfx^{(k+1)} - \bfx^{(k)} \|_{2}.
\end{split}
\end{equation}
Combining~\eqref{eq:subdiff-bound-eq1} and~\eqref{eq:subdiff-bound-eq2} yields:
\[
\begin{split}
	\| \bfv^{(k+1)} \|_{2} & \leq \| \bfv^{(k+1)} - \hat{\bfv}^{(k+1)} \|_{2} + \| \hat{\bfv}^{(k+1)} \|_{2} \\
	& \leq \left(L_{c} + \frac{\beta}{2} (\ve + 2) \right) \|\bfx^{(k+1)} - \bfx^{(k)} \|_{2}.
\end{split}
\]

For the case of $q =1$, denote
\[
\begin{split}
	v_{j}^{(k+1)} &= \sgn(x_{j}^{(k+1)})\phi^{\prime}(|x_{j}^{(k+1)}|) + \frac{1}{\alpha} \left(\sum_{i\in I} \eta_{i} A_{i}\right)_{j}, \forall j \in S^{(\sfK)}, \\
	v_{j}^{(k+1)} &= 0, \forall j \in J \setminus S^{(\sfK)};
\end{split}
\]
\[
\begin{split}
	\hat{v}_{j}^{(k+1)} &= \sgn(x_{j}^{(k+1)})\phi^{\prime}(|x_{j}^{(k)}|) + \frac{1}{\alpha} \left(\sum_{i\in I} \eta_{i} A_{i}\right)_{j}, \forall j \in S^{(\sfK)}, \\
	\hat{v}_{j}^{(k+1)} &= 0, \forall j \in J \setminus S^{(\sfK)}.
\end{split}
\]
In a similar way, we can prove that \eqref{eq:relative-error-condition} holds .
\end{proof}

Finally, we establish our main convergence result. An important tool for establishing the convergence is based on the so-called Kurdyka-\L ojasiewicz (KL) property, which has attracted a lot of attention in recent years. Related preliminaries have been provided in Appendix~\ref{sec:appendix}.

\begin{theorem}
The sequences $\set{\bfx^{(k)}}$  generated by {\rm InISSAPL} converges globally to the limit point $\bfx^{\ast}$, which is a stationary point of $\calE$.
\end{theorem}

\begin{proof}

Since $\set{\bfx^{(k)}}$ is bounded (Lemma~\ref{lem-bound-sufficient-decrease}), there exists a subsequence $(\bfx^{(k_{l})})$ and $\bfx^{\ast}$ such that
\begin{equation}\label{eq:continuity-condition}
  \bfx^{(k_{l})} \to \bfx^{\ast} \text{ and } \calE(\bfx^{(k_{l})}) \to \calE(\bfx^{\ast}),
\text{as } l \to \infty.
\end{equation}

The function $\calE$ satisfies the KL~property~\cite{Attouch2013Convergence}.
Combing~\eqref{eq:sufficient-decrease-condition},~\eqref{lem-relative-error-condition} and~\eqref{eq:continuity-condition}, and by Theorem 2.9 in~\cite{Attouch2013Convergence},  the sequence $\set{\bfx^{(k)}}$ converges globally to the limit point $\bfx^{\ast}$, which is a stationary point of $\calE$.
\end{proof}

\section{Algorithm implementation}
\label{sec:imlpem}

The subproblem in InISSAPL is a weighted $\ell_1$ minimization. Some standard methods like ADMM~\cite{Boyd2011Distributed,Wu2010Augmented,He2012On,Yan2016Self}, split Bregman method~\cite{Goldstein2009Split,Cai2009Convergence,Cai2009Linearized} and primal-dual algorithm~\cite{Chambolle2011first} can be used to efficiently solve it. We here adopt ADMM.
For clarity of description in this section, we refer to $S^{(k)}$, $\phi^{\prime}(|z_{j}^{(k)}|)$, $\bfB^{(k)}$ and $\bfz^{(k)}$ by $S$, $w_{j}$, $\bar{\bfB}$ and $\bar{\bfz}$, respectively. Consequently, \eqref{eq:prob-z} becomes
\begin{equation}\label{eq:alg-supproblem}
  \min_{\bfz}\sum_{ j \in S } w_{j} |z_{j}|	+ \frac{1}{q\alpha} \| \bar{\bfB} \bfz - \bfy \|_{q}^{q} + \frac{\beta}{2}\| \bfz - \bar{\bfz} \|_{2}^{2}.
\end{equation}

We rewrite~\eqref{eq:alg-supproblem} to the following constrained optimization problem:
\begin{equation}\label{eq:alg-supproblem-constrained}
    \begin{aligned}
      &  \min_{\bfz,\bfs,\bft}\sum_{ j \in S } w_{j} |s_{j}|	+ \frac{1}{q\alpha} \| \bft \|_{q}^{q} + \frac{\beta}{2}\| \bfz - \bar{\bfz} \|_{2}^{2}, \\
      & \text{s. t. } \bfz = \bfs,   \bar{\bfB} \bfz  - \bfy = \bft,
    \end{aligned}
\end{equation}
and define the augmented Lagrangian functional for the problem~\eqref{eq:alg-supproblem-constrained} as follows:
\[
\begin{split}
   \mathcal{L}(\bfz,\bfs,\bft;\lambda,\mu) & = \sum_{ j \in S } w_{j} |s_{j}|	+ \frac{1}{q\alpha} \| \bft \|_{q}^{q} + \frac{\beta}{2}\| \bfz - \bar{\bfz} \|_{2}^{2} + \langle \boldsymbol\lambda, \bfz - \bfs \rangle  \\
     & \phantom{=;} + \langle \boldsymbol\mu,  (\bar{\bfB} \bfz - \bfy ) - \bft \rangle + \frac{\gamma}{2} \| \bfz - \bfs \|_{2}^{2} + \frac{\delta}{2} \| (\bar{\bfB} \bfz - \bfy ) - \bft\|_{2}^{2},
\end{split}
\]
where $\gamma, \delta > 0$ are the penalty parameters and $\boldsymbol\lambda, \boldsymbol\mu$ are the Lagrangian multipliers. The ADMM for solving~\eqref{eq:alg-supproblem} is described as follows.

\begin{mdframed}[frametitle = {ADMM: Alternating Direction Method of Multipliers for Solving~\eqref{eq:alg-supproblem}}, frametitlerule = true]


\noindent{\bf Initialization:} Start with $\bfz^{(0)} = \bar{\bfz}, \boldsymbol\lambda^{(0)} = \mathbf{0}, \boldsymbol\mu^{(0)}= \mathbf{0}$.

\noindent {\bf Iteration:} For $l = 0, 1, \ldots, \textmd{MAXit}$,
\begin{enumerate}
\renewcommand{\labelenumi}{ \arabic{enumi}.}
  \item Compute
\begin{equation}\label{eq:admm-s-t}
  (\bfs^{(l+1)}, \bft^{(l+1)}) = \arg \min_{\bfs, \bft} \mathcal{L}(\bfz^{(l)},\bfs,\bft;\boldsymbol\lambda^{(l)},\boldsymbol\mu^{(l)}).
\end{equation}
  \item  Compute
\begin{equation}\label{eq:admm-z}
  \bfz^{(l+1)} = \arg \min_{\bfz} \mathcal{L}(\bfz,\bfs^{(l)},\bft^{(l)};\boldsymbol\lambda^{(l)},\boldsymbol\mu^{(l)}).
\end{equation}
  \item Update
\begin{align}
  \boldsymbol\lambda^{(l+1)} & = \boldsymbol\lambda^{(l)} + \gamma( \bfz^{(l+1)} - \bfs^{(l+1)}),\\
  \boldsymbol\mu^{(l+1)} & = \boldsymbol\mu^{(l)} + \delta( (\bar{\bfB} \bfz^{(l+1)} - \bfy ) -  \bft^{(l+1)} ).
\end{align}
\end{enumerate}
\end{mdframed}

ADMM can solve \eqref{eq:alg-supproblem} to any accuracy. Considering the computational efficiency, we utilize in practice, the following stopping criterion \cite{Boyd2011Distributed}:
\[
    \norm{\tau^{(l+1)}}_2 \leq \sqrt{\sfM}\epsilon^{\mathrm{abs}} + \epsilon^{\mathrm{rel}} \max\set{\norm{\left[
                       \begin{array}{cc}
                         \bfI &  \\
                          & \bar{\bfB} \\
                       \end{array}
                     \right]
\left[
  \begin{array}{c}
    \bfz^{(l+1)} \\
    \bfz^{(l+1)} \\
  \end{array}
\right]
}_{2},
\norm{\left[
        \begin{array}{c}
          -\bfs^{(l+1)} \\
          -\bft^{(l+1)} \\
        \end{array}
      \right]
}_{2},
\norm{\left[
        \begin{array}{c}
          0 \\
          \bfy \\
        \end{array}
      \right]
}_{2}
},
\]
\[
    \norm{\upsilon^{(l+1)}}_{2} \leq \sqrt{\sfN}\epsilon^{\mathrm{abs}} + \epsilon^{\mathrm{rel}}
\norm{
\left[
       \begin{array}{cc}
            \bfI &  \\
             & \bar{\bfB} \\
       \end{array}
    \right]^{T}
\left[
        \begin{array}{c}
           \boldsymbol\lambda^{(l+1)}\\
          \boldsymbol\mu^{(l+1)} \\
        \end{array}
      \right]
}_{2},
\]
where $\tau^{(l+1)} =
\left[
  \begin{array}{c}
    \bfz^{(l+1)} - \bfs^{(l+1)} \\
    \bar{\bfB}\bfz^{(l+1)}  -\bfy -\bft^{(l+1)}\\
  \end{array}
\right]
$,
$\upsilon^{(l+1)} =
\left[
       \begin{array}{cc}
            \bfI &  \\
             & \bar{\bfB} \\
       \end{array}
    \right]^{T}
\left[
  \begin{array}{c}
    \gamma(\bfs^{(l)} - \bfs^{(l+1)}) \\
    \delta(\bft^{(l)} - \bft^{(l+1)}) \\
  \end{array}
\right],
$
are primal and dual residuals, respectively, at the $l$the iteration. $\epsilon^{\mathrm{abs}} > 0$ ia an absolute tolerance and $\epsilon^{\mathrm{rel}}$ is a relative tolerance.

The subproblems~\eqref{eq:admm-s-t} and~\eqref{eq:admm-z} can be  efficiently solved.
\begin{enumerate}
\renewcommand{\labelenumi}{ \arabic{enumi}.}
\renewcommand{\labelenumii}{ \arabic{enumi}.\arabic{enumii}.}
  \item For ~\eqref{eq:admm-s-t}, the minimization with respect to $\bfs$ and $\bft$ is
\[
    \min_{\bfs, \bft}  \sum_{ j \in S } w_{j} |s_{j}|	+ \frac{1}{q\alpha} \| \bft \|_{q}^{q}- \langle \boldsymbol\lambda^{(l)}, \bfs  \rangle
     - \langle \boldsymbol\mu^{(l)}, \bft \rangle + \frac{\gamma}{2} \| \bfz^{(l)} - \bfs  \|_{2}^{2} + \frac{\delta}{2} \| (\bar{\bfB} \bfz^{(l)}  - \bfy ) - \bft \|_{2}^{2},
\]
which can be separated into two independent subproblems.

\begin{enumerate}
  \item The minimization~\eqref{eq:admm-s-t} with respect to $\bfs$
\[
    \min_{\bfs}\sum_{ j \in S } w_{j} |s_{j}|  + \frac{\gamma}{2} \| \bfs - \bfz^{(l)} - \frac{\boldsymbol\lambda^{(l)}}{\gamma} \|_{2}^{2},
\]
has the following closed form solution:
\[
    s_{j}^{(l+1)} = \sgn(z_{j}^{(l)} + \frac{\lambda_{j}^{(l)} }{ \gamma})\max\set{\abs{z_{j}^{(l)} + \frac{\lambda_{j}^{(l)} }{ \gamma} } -  \frac{w_{j} }{ \gamma},0}, \forall j \in S.
\]
  \item The minimization~\eqref{eq:admm-s-t} with respect to $\bft$ is
\[
    \min_{\bft} \frac{1}{q\alpha} \| \bft \|_{q}^{q} + \frac{\delta}{2} \| \bft - (\bar{\bfB} \bfz^{(l)}  - \bfy ) - \frac{\boldsymbol\mu^{(l)}}{\delta}\|_{2}^{2}.
\]
For $q = 1$,
$
  t_{i}^{(l+1)} = \sgn( (\bar{\bfB} \bfz^{(l)} - \bfy )_{i} + \frac{\mu_{i}^{(l)}}{ \delta} )\max\set{\abs{(\bar{\bfB} \bfz^{(l)} - \bfy )_{i}+ \frac{\mu_{i}^{(l)}}{ \delta} } - \frac{1}{ \delta\alpha}, 0 }, \forall i \in I
$. For $q = 2$, $
  t_{i}^{(l+1)} = \frac{\delta \alpha (\bar{\bfB} \bfz^{(l)} - \bfy )_{i} + \alpha \mu_{i}^{(l)}}{1 + \delta\alpha}, \forall i \in I
$. For other $1 < q < \infty$, we can find $t_{j}^{(l+1)}$ via any numerical procedure such as Newton's method.
\end{enumerate}
  \item For~\eqref{eq:admm-z}, the minimization with respect to $\bfz$ is a quadratic optimization problem,
\[
    \min_{\bfz}  \frac{\beta}{2}\| \bfz - \bar{\bfz} \|_{2}^{2} + \langle \boldsymbol\lambda^{(l)}, \bfz \rangle  + \langle \boldsymbol\mu^{(l)}, \bar{\bfB} \bfz  \rangle + \frac{\gamma}{2} \|  \bfz - \bfs^{(l+1)} \|_{2}^{2} + \frac{\delta}{2} \| (\bar{\bfB} \bfz  - \bfy ) - \bft^{(l+1)}\|_{2}^{2}.
\]
Its optimality condition gives a linear system
\begin{equation*}
  (\beta  + \gamma  + \delta \bar{\bfB}^{T}\bar{\bfB}) \bfz = \beta \bar{\bfz} + \gamma \bfs^{(l+1)}  + \delta\bar{\bfB}^{T}(\bfy + \bft^{(l+1)}) -  \boldsymbol\lambda^{(l)} - \bar{\bfB}^{T}\boldsymbol\mu^{(l)},
\end{equation*}
\end{enumerate}
which can be solved efficiently~\cite{Boyd2011Distributed}.

\begin{remark}
  For $q = 2$, we actually only need to introduce one new variable $\bfs$.
\end{remark}

\section{Numerical experiments}
\label{sec:num}

\graphicspath{{figures/}}

In this section, we present numerical experiments to demonstrate the efficiency of the InISSAPL algorithm. All the tests were performed using Windows 10 and  \textsc{Matlab} R2016a 64-bit on a HP Z228 microtower workstation with an Intel(R) Core(TM) i7-4790 CPU @3.60GHz and 8GB memory.

In our experiments, we generated the true signal $\bfx^{o}$ of the sparsity $\kappa$ supported on a random index set with independently and identically distributed Gaussian entries. For the InISSAPL algorithm, we chose $\textmd{MAXit} = 1000$, $\epsilon^{\mathrm{abs}} = 10^{-7}$ and $\epsilon^{\mathrm{abs}} = 10^{-5}$  in the inner ADMM and adopted the following stopping criteria for the outer iteration
\[
    \frac{\norm{\bfx^{(k+1)} - \bfx^{(k)}}_{2}}{\norm{\bfx^{(k)}}_{2}} \leq 10^{-3}.
\]

\subsection{Choice of $p$}
In our first example, we tested the InISSAPL algorithm for $\ell_{p}$-$\ell_2$ minimization to recover sparse vectors with $p$ varying among $\set{0.1,0.3,0.5,0.7,0.9}$. We used a $3000\times5000$ random Gaussian matrix $\bfA$ and a true signal $\bfx^{o}$ of the sparsity $\kappa = 500$. The Gaussian noises with $\sigma = 0.01$ and $\sigma = 0.1$ were added to the clean signal $\bfA \bfx^{o}$ to simulate the measurements $\bfy$. The InISSAPL algorithm was applied to get recovered signals $\bfx^{\ast}$. To show the performance of our algorithm, we chose the results of LASSO as the benchmarks, which is solved by ADMM-lasso~\cite{Boyd2011Distributed}. For ADMM-lasso, we also set $\epsilon^{\mathrm{abs}} = 10^{-7}$ and $\epsilon^{\mathrm{abs}} = 10^{-5}$. We show the relative $L_2$ error $\frac{\norm{\bfx^{\ast} - \bfx^{o}}_{2}}{\norm{\bfx^{\ast}}_{2}}$ in Table~\ref{tab-diff-p-L2-error}. As can be seen, for the low level noise, our InISSAPL algorithm with smaller $p$ generates better results. However, with $p=0.5$, it is more robust to different levels of noise.

\begin{table}[htbp]
  \centering
  \caption{Relative $L_2$ errors of LASSO and InISSAPL for $\ell_{p}$-$\ell_2$ minimization with different $p$.}
    \begin{tabular}{ccccccccc}
    \toprule
          &       & LASSO        & $p = 0.1$        & $p = 0.3$        & $p = 0.5$        & $p = 0.7$        & $p = 0.9$ \\
\cmidrule{3-7}    $\kappa=500$ &       & $L_2$ Error        & $L_2$ Error        & $L_2$ Error        & $L_2$ Error        & $L_2$ Error        & $L_2$ Error \\
    \midrule
    $\sigma=0.01$ &       &   0.05633           &    0.01379          &    0.01669          &    0.01719          &    0.01589          &  0.02118\\
    $\sigma=0.1$ &       &   0.26325           &   0.20197           &   0.17989           &    0.17310          &    0.18830          &  0.21276\\
    \bottomrule
    \end{tabular}%
  \label{tab-diff-p-L2-error}%
\end{table}%

\subsection{Robust recovery from measurements with heavy-tailed noise}
In the third example, we tested our InISSAPL algorithm for $\ell_{p}$-$\ell_{1}$ minimization. We chose $p = 0.5$. We generated a $256\times512$ random Gaussian $\bfA$ and a true signal $\bfx^{o}$ of the sparsity $\kappa = 25$. The $\bfA \bfx^{o}$ was corrupted by impulsive noise obeying the standard $(0,1)$-Cauchy distribution, which is scaled by a factor of $10^{-2}$. Figure~\ref{fig-Cauchy}(a) shows both the noiseless and noisy observations. The noisy observed signal approximates closely the noiseless observation almost everywhere except two outliers at the entries 124 and 249. Figure~\ref{fig-Cauchy}(b)-(d) show the recovered signals by LASSO, InISSAPL for $\ell_{p}$-$\ell_{2}$ minimization, and InISSAPL for $\ell_{p}$-$\ell_{1}$ minimization. Note that LASSO and $\ell_{p}$-$\ell_{2}$ minimization are very sensitive to the outliers and failed to reconstruct the signal. As can be seen in Figure~\ref{fig-Cauchy}(d), the $\ell_{p}$-$\ell_{1}$ minimization is able to recover the sparse signal with high accuracy.

\begin{figure}[htbp]
 \captionsetup[subfigure]{justification=centering}
  \centering
  \begin{tabular}{c@{\hspace{1mm}}c@{\hspace{1mm}}}
  \subfloat[][Noiseless and noisy observed signals]
  {\includegraphics[width=0.50\textwidth]{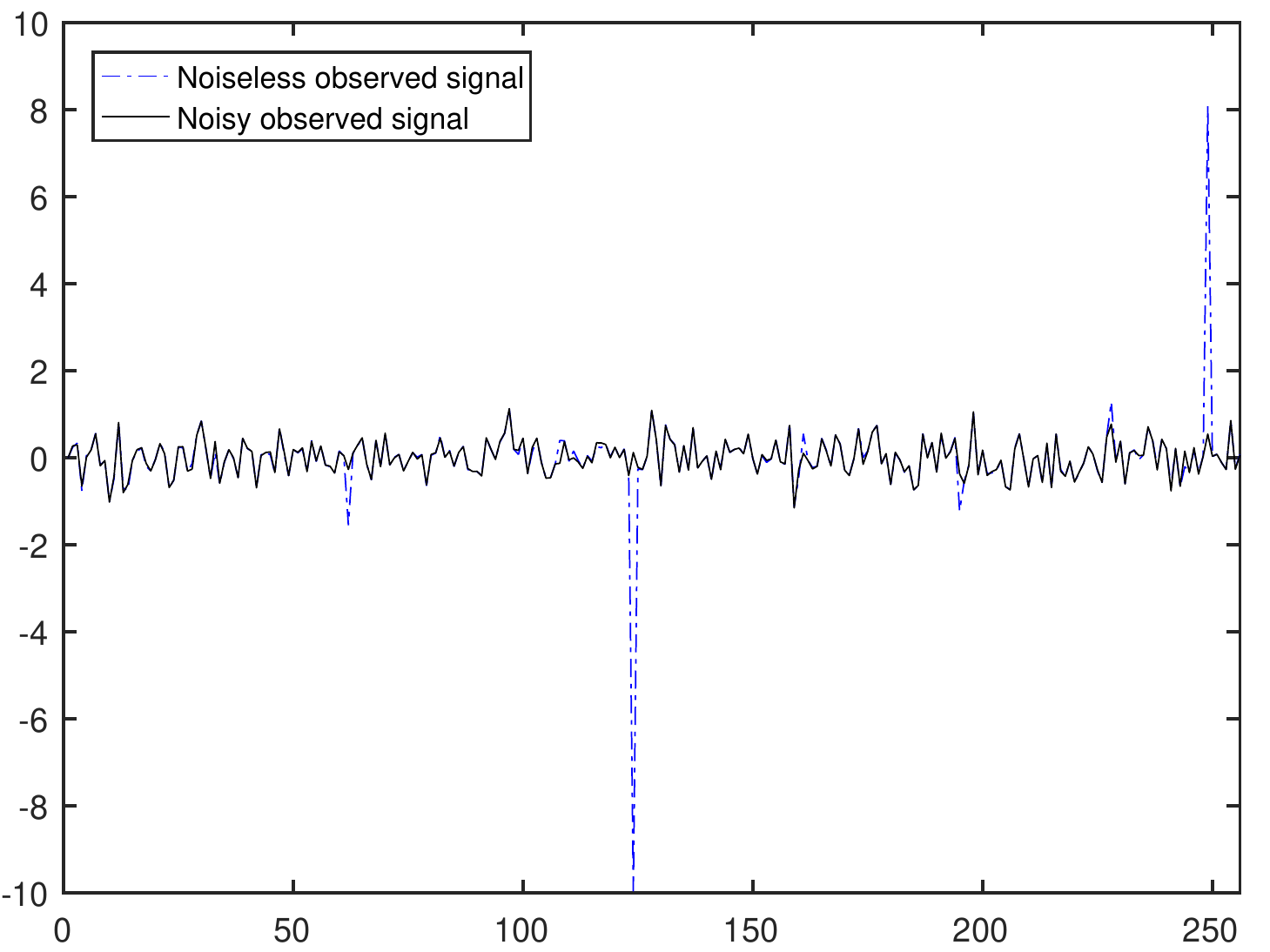}}&
  \subfloat[][Recovered signal by LASSO]
  {\includegraphics[width=0.50\textwidth]{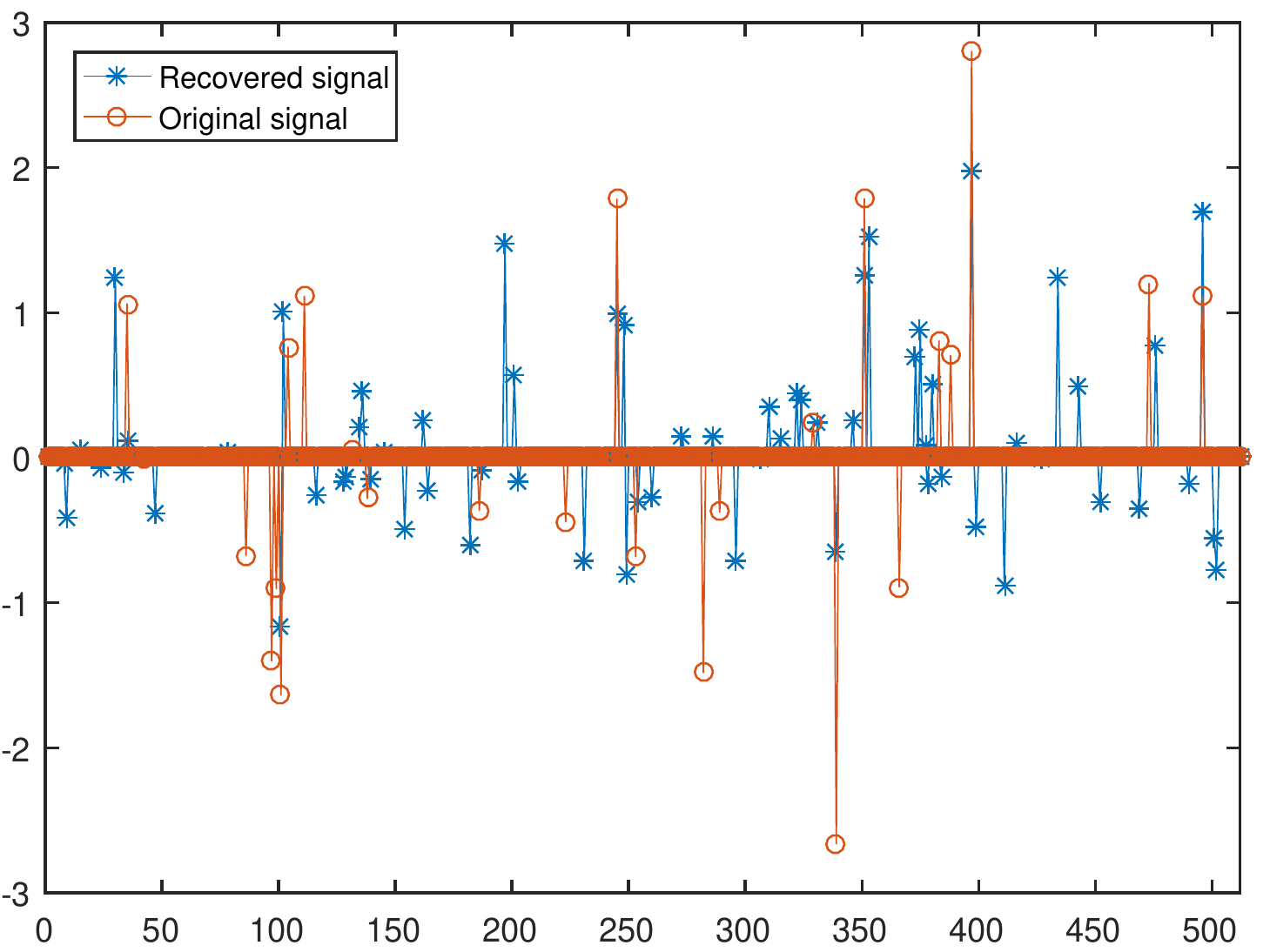}}\\
  \subfloat[][Recovered signal by $\ell_{p}$-$\ell_{2}$  via InISSAPL]
  {\includegraphics[width=0.50\textwidth]{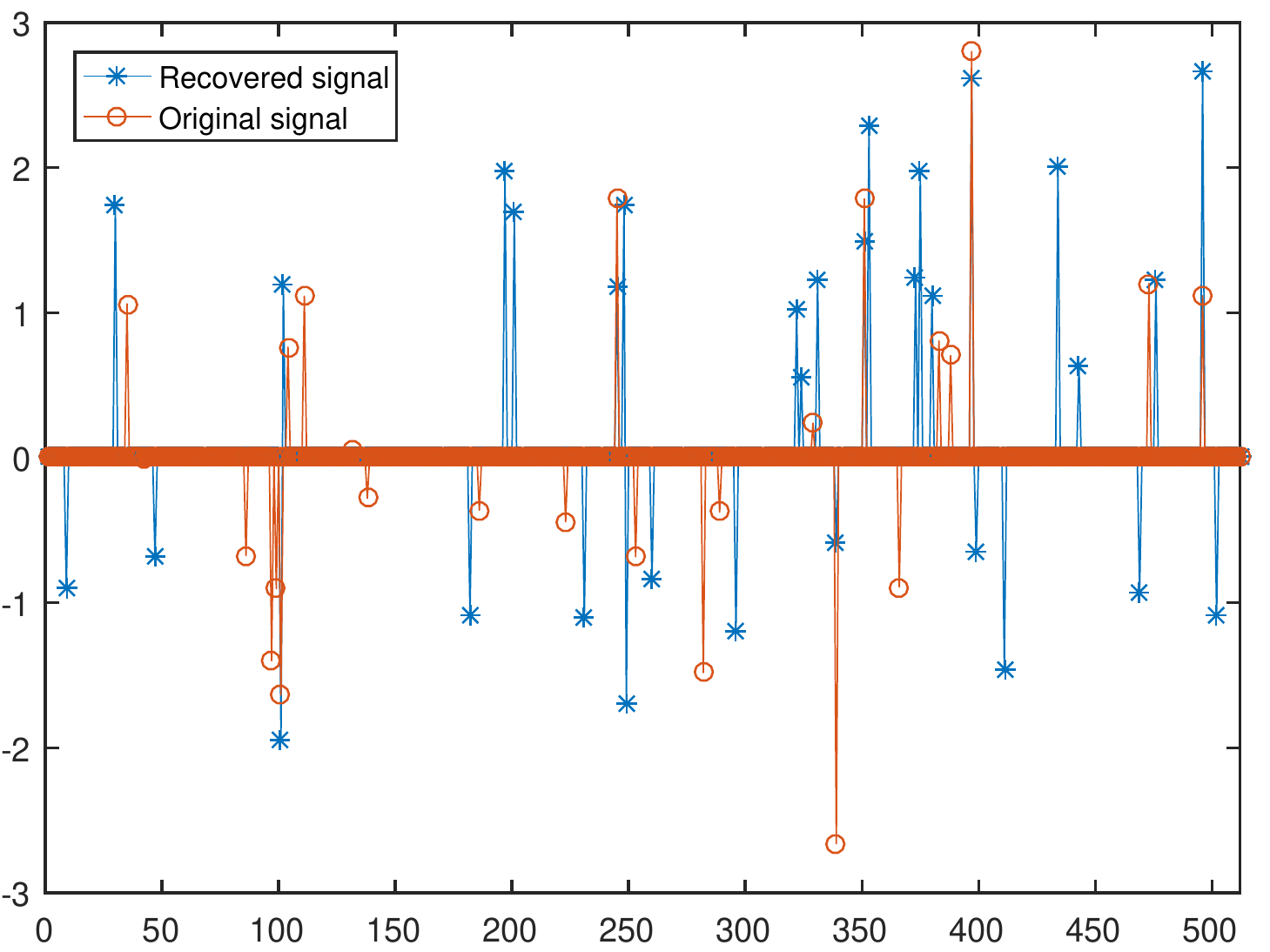}}&
  \subfloat[][Recovered signal by $\ell_{p}$-$\ell_{1}$  via InISSAPL]
  {\includegraphics[width=0.50\textwidth]{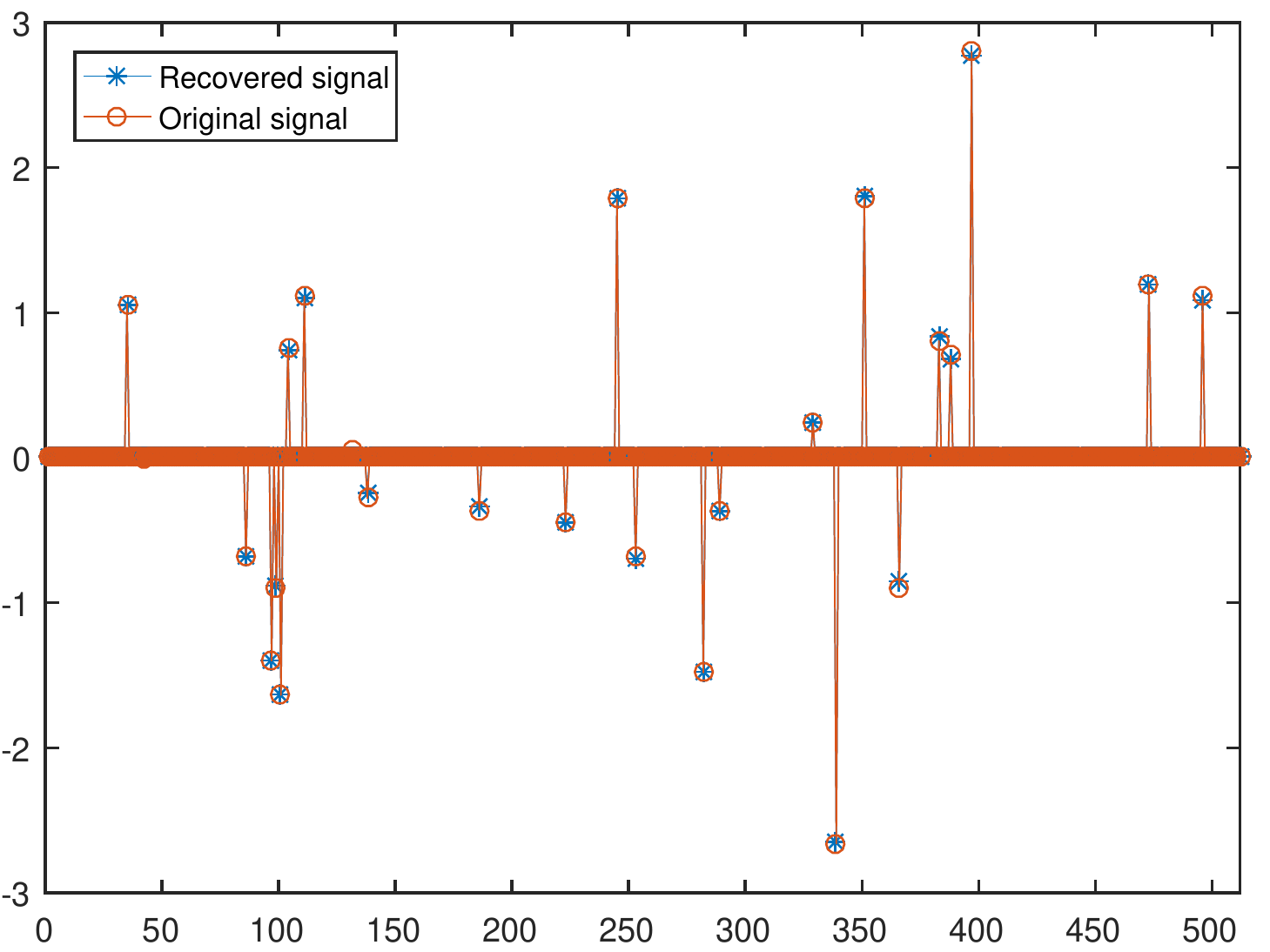}}
   \end{tabular}
  \caption{Reconstructions of a sparse signal from measurements corrupted by impulsive noise using different methods. $\ell_{p}$-$\ell_{1}$ minimization via InISSAPL is more robust to this case.}
  \label{fig-Cauchy}
\end{figure}

\section{Conclusions}
\label{sec:conclusion}

We proposed an iterative support shrinking algorithm for non-Lipschtiz $\ell_{p}$-$\ell_{q}$ minimization. The proposed algorithm overcomes the non-Lipschtizian by iteratively adding constraints on the support to the original problem. It is a new type of reweighted $\ell_{1}$ algorithm. The algorithm is easy to implement. The subproblem in each iteration is solved inexactly by ADMM. We proved the global convergence of the iterative sequence, whose limit is a stationary point of the $\ell_{p}$-$\ell_{q}$ objective function. We also showed a more practical lower bound theory of the iterates. Numerical experiments demonstrated the performance of the algorithm. Due to the successive size reduction of the subproblem, it has good potentials in applications for large scale sparse signal recovery problems.

\section{Appendix}
\label{sec:appendix}

We recall some definitions and results here.

\begin{definition}[Subdifferentials~\cite{Rockafellar1998Variational}]\label{def-subdiff}
  Let $h: \mathbb{R}^{\sfN} \to \mathbb{R}\cup\{+\infty\}$ be a proper, lower semicontinuous function.
  \begin{enumerate}
    \item The \emph{regular subdifferential} of $h$ at $\bar{\bfx} \in \dom h = \{\bfx \in \mathbb{R}^{\sfN}: h(\bfx) < +\infty  \}$ is defined as
    \[
    \widehat{\partial}h(\bar{\bfx} ) := \set{\bfv \in\mathbb{R}^{\sfN}:
    \liminf_{\substack{ \bfx \to \bar{\bfx} \\ \bfx \neq \bar{\bfx} }}\frac{h(\bfx)-h(\bar{\bfx})- \langle \bfv,\bfx -\bar{\bfx} \rangle}{\|\bfx-\bar{\bfx}\|}\geq 0
    };
    \]
    \item The (limiting) \emph{subdifferential} of $h$ at $\bar{\bfx}  \in \dom h $ is defined as
     \[
    \partial h(\bar{\bfx} ):=\set{\bfv \in\mathbb{R}^{\sfN}: \exists \bfx^{(k)} \to \bar{\bfx} , h(\bfx^{(k)}) \to h(\bfx), \bfv^{(k)}\in \widehat{\partial} h(\bfx^{(k)}), \bfv^{(k)} \to \bfv
}.
    \]
  \end{enumerate}
\end{definition}

\begin{remark}
Form Definition~\ref{def-subdiff}, the following properties hold:
\begin{enumerate}
  \item For any $\bar{\bfx}  \in \dom h $, $\widehat{\partial}h(\bar{\bfx} ) \subset \partial h(\bar{\bfx} )$. If $h$ is continuously differentiable at $\bar{\bfx} $, then $\widehat{\partial}h(\bar{\bfx} ) = \partial h(\bar{\bfx} )= \set{\nabla h(\bar{\bfx} )}$;
  \item For any $\bar{\bfx}  \in \dom h $, the subdifferential set $\partial h(\bar{\bfx} )$ is closed, i.e,
\[
\set{\bfv \in\mathbb{R}^{\sfN}:\exists \bfx^{(k)} \to \bar{\bfx}, h(\bfx^{(k)}) \to h(\bar{\bfx}), \bfv^{(k)}\in\partial h(\bfx^{(k)}), \bfv^{(k)} \to \bfv }\subset \partial h(\bar{\bfx} ).
\]
\end{enumerate}
\end{remark}

The foundational works on the Kurdyka-\L ojasiewicz (KL) property property are
due to \L ojasiewicz~\cite{Lojasiewicz1963Une} and Kurdyka~\cite{Kurdyka1998gradients}. For the development of the appliciation of KL property in optimization theory, see~\cite{Bolte2006ojasiewicz,Attouch2010Proximal,Attouch2013Convergence,Bolte2014Proximal}
and reference therein.

\begin{definition}[Kurdyka-\L ojasiewicz Property~\cite{Attouch2010Proximal}]
    A proper function $h$ is said to have the \emph{Kurdyka-\L ojasiewicz property} at $\bar{\bfx} \in \dom \partial h = \{\bfx \in \mathbb{R}^{\sfN}: \partial h(\bfx) \neq \emptyset\}$ if there exist $\zeta \in (0, +\infty]$, a neighborhood $U$ of $\bar{\bfx}$, and a continuous concave function $\varphi: [0, \zeta) \to \mathbb{R}_{+}$ such that
    \begin{enumerate}
      \item $\varphi(0) = 0$;
      \item $\varphi(0)$ is $C^{1}$ on $(0,\zeta)$;
      \item for all $s \in (0,\zeta)$, $\varphi^{\prime}(s) > 0$;
      \item for all $\bfx \in U$ satisfying $h(\bar{\bfx}) < h(\bfx) < h(\bar{\bfx}) + \zeta$, the Kurdyka-\L ojasiewicz inequality holds:
      \[
        \varphi^{\prime}(h(\bfx) - h(\bar{\bfx})) \dist (0,\partial h(\bfx)) \geq 1.
      \]
      where $\dist (0,\partial h(\bfx)) = \min\{\| \bfv \|: \bfv \in \partial h(\bfx) \}$,
    \end{enumerate}
\end{definition}

A proper, lower semicontinuous function $h$ satisfying the KL property at all points in $\dom \partial h$ is called a \emph{KL function}. On can refer to~\cite{Attouch2010Proximal,Attouch2013Convergence,Bolte2014Proximal} for examples of KL functions. For this paper, the function $\calE$ satisfies the KL~property~\cite{Attouch2013Convergence}.


\begin{thebibliography}{10}

\bibitem{Attouch2010Proximal}
H.~Attouch, J.~Bolte, P.~Redont, and A.~Soubeyran.
\newblock Proximal alternating minimization and projection methods for
  nonconvex problems: An approach based on the {Kurdyka-{\L}ojasiewicz}
  inequality.
\newblock {\em Math. Oper. Res.}, 35(2):438--457, April 30 2010.

\bibitem{Attouch2013Convergence}
H.~Attouch, J.~Bolte, and B.~F. Svaiter.
\newblock Convergence of descent methods for semi-algebraic and tame problems:
  Proximal algorithms, forward-backward splitting, and regularized
  {{Gauss}}-{{Seidel}} methods.
\newblock {\em Math. Program.}, 137(1-2):91--129, 2013.

\bibitem{Beck2009fast}
A.~Beck and M.~Teboulle.
\newblock A fast iterative shrinkage-thresholding algorithm for linear inverse
  problems.
\newblock {\em SIAM J. Imaging Sci.}, 2(1):183--202, 2009.

\bibitem{Bian2013Worst}
W.~Bian and X.~Chen.
\newblock Worst-case complexity of smoothing quadratic regularization methods
  for non-{Lipschitzian} optimization.
\newblock {\em SIAM J. Optim.}, 23(3):1718--1741, 2013.

\bibitem{Bolte2006ojasiewicz}
J.~Bolte, A.~Daniilidis, and A.~Lewis.
\newblock The {{\L}ojasiewicz} inequality for nonsmooth subanalytic functions
  with applications to subgradient dynamical systems.
\newblock {\em SIAM J. Optim.}, 17(4):1205--1223, 2006.

\bibitem{Bolte2014Proximal}
J.~B. Bolte, S.~Sabach, and M.~Teboulle.
\newblock Proximal alternating linearized minimization for nonconvex and
  nonsmooth problems.
\newblock {\em Math. Program.}, 146(1-2):459--494, 2014.

\bibitem{Boyd2011Distributed}
S.~Boyd, N.~Parikh, E.~Chu, B.~Peleato, and J.~Eckstein.
\newblock Distributed optimization and statistical learning via the alternating
  direction method of multipliers.
\newblock {\em Found. Trends Mach. Learn.}, 3(1):1--122, Jan. 2011.

\bibitem{Bredies2015Minimization}
K.~Bredies, D.~A. Lorenz, and S.~Reiterer.
\newblock Minimization of non-smooth, non-convex functionals by iterative
  thresholding.
\newblock {\em J. Optim. Theory Appl.}, 165(1):78--112, Apr. 2015.

\bibitem{Bruckstein2009From}
A.~M. Bruckstein, D.~L. Donoho, and M.~Elad.
\newblock From sparse solutions of systems of equations to sparse modeling of
  signals and images.
\newblock {\em SIAM Rev.}, 51(1):34--81, 2009.

\bibitem{Cai2009Convergence}
J.~Cai, S.~Osher, and Z.~Shen.
\newblock Convergence of the linearized bregman iteration for {$\ell_1$}-norm
  minimization.
\newblock {\em Math. Comput.}, 78(268):2127--2136, 2009.

\bibitem{Cai2009Linearized}
J.~Cai, S.~Osher, and Z.~Shen.
\newblock Linearized bregman iterations for compressed sensing.
\newblock {\em Math. Comput.}, 78(267):1515--1536, 2009.

\bibitem{Candes2006Robust}
E.~J. Cand{\`e}s, J.~Romberg, and T.~Tao.
\newblock Robust uncertainty principles: exact signal reconstruction from
  highly incomplete frequency information.
\newblock {\em IEEE Trans. Inform. Theory}, 52(2):489--509, Feb 2006.

\bibitem{Candes2008Enhancing}
E.~J. Cand{\`e}s, M.~B. Wakin, and S.~P. Boyd.
\newblock Enhancing sparsity by reweighted $\ell_1$ minimization.
\newblock {\em J. Fourier Anal. Appl.}, 14(5):877--905, Dec 2008.

\bibitem{Chambolle2011first}
A.~Chambolle and T.~Pock.
\newblock A first-order primal-dual algorithm for convex problems with
  applications to imaging.
\newblock {\em J. Math. Imaging Vision}, 40(1):120--145, 2011.

\bibitem{Chartrand2008Restricted}
R.~Chartrand and V.~Staneva.
\newblock Restricted isometry properties and nonconvex compressive sensing.
\newblock {\em Inverse Problems}, 24(3):035020, 14, 2008.

\bibitem{Chartrand2008Iteratively}
R.~Chartrand and W.~Yin.
\newblock Iteratively reweighted algorithms for compressive sensing.
\newblock In {\em Proc. IEEE Int. Conf. Acoustics, Speech and Signal
  Processing}, pages 3869--3872, March 2008.

\bibitem{Chen2012Smoothing}
X.~Chen.
\newblock Smoothing methods for nonsmooth, nonconvex minimization.
\newblock {\em Math. Program.}, 134(1):71--99, 2012.

\bibitem{Chen2013Optimality}
X.~Chen, L.~Niu, and Y.~Yuan.
\newblock Optimality conditions and a smoothing trust region newton method for
  non{Lipschitz} optimization.
\newblock {\em SIAM J. Optim.}, 23(3):1528--1552, July 2013.

\bibitem{Chen2010Lower}
X.~Chen, F.~Xu, and Y.~Ye.
\newblock Lower bound theory of nonzero entries in solutions of
  $\ell_2$-$\ell_p$ minimization.
\newblock {\em SIAM J. Sci. Comput.}, 32(5):2832--2852, 2010.

\bibitem{Chen2014Convergence}
X.~Chen and W.~Zhou.
\newblock Convergence of the reweighted $\ell_1$ minimization algorithm for
  $\ell_2$-$\ell_p$ minimization.
\newblock {\em Comput. Optim. Appl.}, 59(1):47--61, Oct 2014.

\bibitem{Daubechies2004iterative}
I.~Daubechies, M.~Defrise, and C.~De~Mol.
\newblock An iterative thresholding algorithm for linear inverse problems with
  a sparsity constraint.
\newblock {\em Comm. Pure Appl. Math.}, 57(11):1413--1457, 2004.

\bibitem{Daubechies2010Iteratively}
I.~Daubechies, R.~A. Devore, M.~Fornasier, and C.~S. Güntürk.
\newblock Iteratively reweighted least squares minimization for sparse
  recovery.
\newblock {\em Comm. Pure Appl. Math.}, 63(1):1--38, 2010.

\bibitem{Dielman2005Least}
T.~E. Dielman.
\newblock Least absolute value regression: recent contributions.
\newblock {\em J. Stat. Comput. Simul.}, 75(4):263--286, 2005.

\bibitem{Donoho2006Compressed}
D.~L. Donoho.
\newblock Compressed sensing.
\newblock {\em IEEE Trans. Inform. Theory}, 52(4):1289--1306, 2006.

\bibitem{Foucart2009Sparsest}
S.~Foucart and M.-J. Lai.
\newblock Sparsest solutions of underdetermined linear systems via
  $\ell_q$-minimization for $0<q \leq 1$.
\newblock {\em Appl. Comput. Harmon. Anal.}, 26(3):395 -- 407, May 2009.

\bibitem{Goldstein2009Split}
T.~Goldstein and S.~Osher.
\newblock The split {Bregman} method for {L1}-regularized problems.
\newblock {\em SIAM J. Imaging Sci.}, 2(2):323--343, 2009.

\bibitem{Gorodnitsky1997Sparse}
I.~F. Gorodnitsky and B.~D. Rao.
\newblock Sparse signal reconstruction from limited data using {{FOCUSS}}: a
  re-weighted minimum norm algorithm.
\newblock {\em IEEE Trans. Signal Process.}, 45(3):600--616, Mar 1997.

\bibitem{He2012On}
B.~He and X.~Yuan.
\newblock On the {$O(1/n)$} convergence rate of the douglas�crachford
  alternating direction method.
\newblock {\em SIAM J. Numer. Anal.}, 50(2):700--709, 2012.

\bibitem{Huber1981Robust}
H.~P. J.
\newblock {\em Robust statistics}.
\newblock Wiley, New York, 1981.

\bibitem{Krishnan2009Fast}
D.~Krishnan and R.~Fergus.
\newblock Fast image deconvolution using hyper-laplacian priors.
\newblock In {\em Proc. 22nd Int. Conf. Neural Information Processing Systems},
  pages 1033--1041, 2009.

\bibitem{Kurdyka1998gradients}
K.~Kurdyka.
\newblock On gradients of functions definable in o-minimal structures.
\newblock {\em Ann. Inst. Fourier (Grenoble)}, 48(3):769--783, 1998.

\bibitem{Lai2011Unconstrained}
M.-J. Lai and J.~Wang.
\newblock An unconstrained $\ell_q$ minimization with $0 < q \leq 1 $ for
  sparse solution of underdetermined linear systems.
\newblock {\em SIAM J. Optim.}, 21(1):82--101, 2011.

\bibitem{Lai2013Improved}
M.-J. Lai, Y.~Xu, and W.~Yin.
\newblock Improved iteratively reweighted least squares for unconstrained
  smoothed $\ell_q$ minimization.
\newblock {\em SIAM J. Numer. Anal.}, 51(2):927--957, 2013.

\bibitem{Lojasiewicz1963Une}
S.~{\L}ojasiewicz.
\newblock Une propri\'et\'e topologique des sous-ensembles analytiques r\'eels.
\newblock In {\em Les \'Equations aux {D}\'eriv\'ees {P}artielles ({P}aris,
  1962)}, pages 87--89. \'Editions du Centre National de la Recherche
  Scientifique, Paris, 1963.

\bibitem{Lu2014Iterative}
Z.~Lu.
\newblock Iterative reweighted minimization methods for $\ell_p$ regularized
  unconstrained nonlinear programming.
\newblock {\em Math. Program.}, 147(1):277--307, Oct 2014.

\bibitem{Lv2009unified}
J.~Lv and Y.~Fan.
\newblock A unified approach to model selection and sparse recovery using
  regularized least squares.
\newblock {\em Ann. Statist.}, 37(6A):3498--3528, 2009.

\bibitem{Nataraj1995Sparse}
B.~K. Natarajan.
\newblock Sparse approximate solutions to linear systems.
\newblock {\em SIAM J. Comput.}, 24(2):227--234, 1995.

\bibitem{Paredes2011Compressive}
J.~L. Paredes and G.~R. Arce.
\newblock Compressive sensing signal reconstruction by weighted median
  regression estimates.
\newblock {\em IEEE Trans. Signal Process.}, 59(6):2585--2601, 2011.

\bibitem{Rockafellar1998Variational}
R.~T. Rockafellar and R.~J.-B. Wets.
\newblock {\em Variational Analysis}, volume 317 of {\em Grundlehren der
  Mathematischen Wissenschaften}.
\newblock Springer-Verlag Berlin Heidelberg, 1998.

\bibitem{Sun2012Recovery}
Q.~Sun.
\newblock Recovery of sparsest signals via $\ell_q$-minimization.
\newblock {\em Appl. Comput. Harmon. Anal.}, 32(3):329 -- 341, May 2012.

\bibitem{Tropp2010Computational}
J.~A. Tropp and S.~J. Wright.
\newblock Computational methods for sparse solution of linear inverse problems.
\newblock {\em Proc. IEEE}, 98(6):948--958, June 2010.

\bibitem{Wu2010Augmented}
C.~Wu and X.-C. Tai.
\newblock Augmented {Lagrangian} method, dual methods, and split {Bregman}
  iteration for {ROF}, vectorial {TV}, and high order models.
\newblock {\em SIAM J. Imaging Sci.}, 3(3):300--339, 2010.

\bibitem{Xu2012L12}
Z.~Xu, X.~Chang, F.~Xu, and H.~Zhang.
\newblock {$L_{1/2}$} regularization: A thresholding representation theory and
  a fast solver.
\newblock {\em IEEE Trans. Neural Netw. Learn. Syst.}, 23(7):1013--1027, Jul
  2012.

\bibitem{Yan2016Self}
M.~Yan and W.~Yin.
\newblock Self equivalence of the alternating direction method of multipliers.
\newblock In R.~Glowinski, S.~J. Osher, and W.~Yin, editors, {\em Splitting
  Methods in Communication, Imaging, Science, and Engineering. Scientific
  Computation}. Springer, Cham, 2016.

\bibitem{Zou2008One}
H.~Zou and R.~Li.
\newblock One-step sparse estimates in nonconcave penalized likelihood models.
\newblock {\em Ann. Statist.}, 36(4):1509--1533, Aug. 2008.

\bibitem{Zuo2013Generalized}
W.~Zuo, D.~Meng, L.~Zhang, X.~Feng, and D.~Zhang.
\newblock A generalized iterated shrinkage algorithm for non-convex sparse
  coding.
\newblock In {\em Proc. IEEE Int. Conf. Computer Vision}, pages 217--224, Dec
  2013.

\end{thebibliography}
\end{document}